\documentclass[12pt]{amsart}   
\usepackage{amssymb}   
\usepackage{colordvi}   
\usepackage{graphicx}   
\usepackage{vmargin}   
\setpapersize{USletter}   
\setmargrb{1in}{0.95in}{1in}{0.95in} % --- sets all four margins.  Cool.   
\numberwithin{equation}{section}   
\newtheorem{theorem}{Theorem}[section]   
\newtheorem{proposition}[theorem]{Proposition}   
\newtheorem{lemma}[theorem]{Lemma}   
\newtheorem{corollary}[theorem]{Corollary}   
   
\newtheorem{defn}[theorem]{Definition}   
\newtheorem{question}{Question}   
   
\theoremstyle{definition}  
  
\newtheorem{example}[theorem]{Example}   
\newtheorem{remark}[theorem]{Remark}

\PII{}  
\copyrightinfo{}{}   
\newcounter{FNC}[page]   
\def\fauxfootnote#1{{\addtocounter{FNC}{2}\Magenta{$^\fnsymbol{FNC}$}%   
     \let\thefootnote\relax\footnotetext{\Magenta{$^\fnsymbol{FNC}$#1}}}}   
\DeclareMathAlphabet{\mathpzc}{OT1}{pzc}{m}{it}   
  
\newcommand{\K}{{\mathbb K}}   
\newcommand{\R}{{\mathbb R}}   
\newcommand{\C}{{\mathbb C}}   
\newcommand{\N}{{\mathbb N}}   
   
\newcommand{\Q}{{\mathbb Q}}   
\newcommand{\PP}{{\mathbb P}}   
\newcommand{\HH}{{\mathbb H}}  
  
\newcommand\maj{\unlhd}  
\newcommand\SHsym{\mathcal{O}}                   % Schur-Horn orbitope  
\newcommand\Sfunc{\mathcal{L}}                % Schur functor  
\newcommand\SH{\SHsym}  
\newcommand{\stab}{\mbox{\rm stab}}           % Affine Hull  
\newcommand{\cara}{\mathfrak{c}}              % Caratheodory number  
\newcommand{\BPi}{\Pi^s}  
\newcommand\s{\sigma}                         % sigma  
\newcommand\ts{\tilde{\sigma}}                % sigma tilde  
  
\newcommand\Sym{\mathrm{Sym}}                 % symmetric power  
\newcommand\sym{\Sym_2\R^n}                   % symmetric matrices  
\renewcommand\skew{\wedge_2\R^n}              % skew-symmetric matrices  
  
\newcommand{\aff}{\mbox{\rm aff}}              % Affine Hull  
\newcommand{\conv}{\mbox{\rm conv}}            % Convex Hull  
\newcommand{\Tr}{\mbox{\rm Tr}}                % Trace  
\newcommand{\Hom}{\mbox{\rm Hom}}              % Homomorphisms  
\newcommand{\End}{\mbox{\rm End}}              % Endomorphisms  
\newcommand{\SEnd}{\mbox{\scriptsize\rm End}}  % Scriptsize \End.  
  
\newcommand{\bdt}{\mbox{\huge\rm .}}

\newcommand{\calO}{{\mathcal O}}  
  
\newcommand\Cara{\mathcal{C}}               % Caratheodory orbitope   
   
\newcommand\h{h}                            % support function   
\renewcommand\l{\lambda}                      % l  
\newcommand\hl{\tilde{\lambda}}                      % l  
\newcommand\sphere{\mathbb{S}}              % sphere  
\newcommand\hcube{\mathsf{HC}}              % Halfcube   
\newcommand\diag{\mathsf{D}}                % diagonal map   
\newcommand\Sdiag{\mathsf{SD}}               % skew-diagonal map   
\newcommand\inner[1]{\langle #1 \rangle}    % inner product   
\title{\larger Orbitopes}   
   
\author{Raman Sanyal}   
\address{Raman Sanyal\\   
	Department of Mathematics\\   
        University of California\\   
        Berkeley, California 94720 \\   
        USA}   
\email{sanyal@math.berkeley.edu}   
\thanks{Research of Sanyal supported by a Miller Research Fellowship 
at UC Berkeley} 
\author{Frank Sottile}   
\address{Frank Sottile\\   
         Department of Mathematics\\   
         Texas A\&M University\\   
         College Station\\   
         Texas \ 77843\\   
         USA}   
\email{sottile@math.tamu.edu}   
\thanks{Research of Sottile supported in part by NSF grants DMS-070105 and DMS-1001615}   
\author{Bernd Sturmfels}   
\address{Bernd Sturmfels\\   
	Department of Mathematics\\   
        University of California\\   
        Berkeley, California 94720 \\   
        USA}   
\email{bernd@math.berkeley.edu}   
\thanks{Research of Sturmfels supported in part by NSF grant DMS-0757207}   
\begin{document}

\begin{abstract}   
 An orbitope is the convex hull of an orbit of a compact group acting    
 linearly on a vector space.  These highly symmetric convex bodies lie  
at the crossroads of several fields including 
 convex geometry, algebraic geometry,  and optimization.   
We present a self-contained theory of orbitopes with particular  
emphasis on instances arising from  
the groups $SO(n)$ and $O(n)$.  
These include  
Schur-Horn orbitopes, tautological orbitopes,  
Carath\'eodory orbitopes, Veronese orbitopes, and Grassmann orbitopes.  
We study their face lattices, their algebraic  
boundaries, and  
representations as spectrahedra or projected spectrahedra.  
\end{abstract}   
   
\maketitle   
%%%%%%%%%%%%%%%%%%%%%%%%%%%%%%%%%%%%%%%%%%%%%%%%%%%%%%%%%%%%%%%%%%%%%%%%%%%%%   
  
\section{Introduction}  
  
An \emph{orbitope} is the convex hull of an orbit   
of a compact algebraic group $G$ acting linearly on a real vector space.  
The orbit has the structure of a real algebraic  
variety, and the orbitope is a convex, semi-algebraic set. Thus, the study of  
algebraic orbitopes lies at the heart of \emph{convex algebraic geometry} --  
the fusion of convex geometry and (real) algebraic geometry.  
  
Orbitopes have appeared in many contexts in mathematics and its applications. 
Orbitopes of finite groups are highly  symmetric convex 
polytopes which include the platonic solids, permutahedra, Birkhoff 
polytopes, and other favorites from Ziegler's text book \cite{Ziegler}, as well as the
Coxeter orbihedra studied by McCarthy, Ogilvie, Zobin, and Zobin~\cite{Zobin}.
Farran and Robertson's regular convex bodies~\cite{FR94} are orbitopal generalzations of
regular polytopes, which were classified by Madden and Robertson~\cite{MR95}.
%
%Orbitopes gives rise to generalizations of regular polytopes as shown by 
%Farran and Robertson \cite{FR94}. See Madden and Robertson \cite{MR95} for 
%a classification of regular convex bodies.  
%
Orbitopes for compact Lie groups, such as $SO(n)$, have 
appeared in investigations ranging from  protein structure 
prediction~\cite{LSS08} and quantum information~\cite{AII06} to calibrated 
geometries~\cite{HL1982}.  Barvinok and Blekherman  studied the volumes of 
convex bodies dual to certain $SO(n)$-orbitopes, and they concluded that there 
are many more  non-negative polynomials than sums of squares \cite{BB}.  
 
This paper initiates the study of orbitopes as geometric objects in  
their own right. The questions we ask about orbitopes originate  
from  three different perspectives: {\em convexity},   
{\em algebraic geometry}, and {\em optimization}.  
In convexity, one would seek to characterize all  faces of an orbitope.  
In algebraic geometry, one would examine the  
Zariski closure of its boundary and identify  
the  components and singularities of that hypersurface.  
In optimization, one would ask whether the orbitope is a spectrahedron  
or the projection of a spectrahedron.  
  
Spectrahedra are to semidefinite  
programming what polyhedra are to linear programming.   
More precisely, a {\em spectrahedron} is  
the intersection of  the cone of positive semidefinite  
matrices with an affine space. It can be represented  
as the    set of points $x \in \R^n$ such that  
\begin{equation}\label{Eq:LMI}   
    A_0 + x_1 A_1 +  \dotsb + x_n A_n  
\, \,   \succeq\ \, 0\,,   
\end{equation}   
where $A_0,A_1,\ldots, A_n$ are symmetric matrices and $\succeq 0$ denotes positive  
semidefiniteness.   
From a spectrahedral description many geometric properties,  
both convex and algebraic, are within reach. Furthermore,  
if an orbitope admits a representation (\ref{Eq:LMI})  
then it is easy to maximize or minimize a linear function over that orbitope.  
Here is a simple illustration.  
  
\begin{example} \label{ex:HankelSmall}  
Consider the action of the group $G = SO(2)$ on the space  
${\rm Sym}_4(\R^2) \simeq \R^5$ of binary quartics and take the convex hull  
of the orbit of $v = x^4$.   
The four-dimensional convex body ${\rm conv}(G \cdot v)$  
is a {\em Carath\'eodory  orbitope}. This orbitope 
 is a spectrahedron: it coincides with the set of all binary quartics   
$\,\lambda_0 x^4 + 4 \lambda_1 x^3 y + 6 \lambda_2 x^2 y^2  
+ 4 \lambda_3 x y^3 + \lambda_4 y^4\,$ such that 
\begin{equation} \qquad  
\label{3by3hankel}  \begin{pmatrix}  
 \lambda_0 & \lambda_1 & \lambda_2 \\  
 \lambda_1 & \lambda_2 & \lambda_3 \\  
 \lambda_2 & \lambda_3 & \lambda_4   
 \end{pmatrix} \,\, \succeq \,\, 0   
 \qquad \hbox{and} \qquad  
\lambda_0 + 2 \lambda_2 + \lambda_4 = 1.  
\end{equation}  
This representation (\ref{3by3hankel}) will be derived in Section 5, 
where we will also see that it is equivalent to classical results from 
 the theory of positive polynomials \cite{Rez}.  
The Hankel matrix shows that the boundary of ${\rm conv}(G \cdot v)$ is an irreducible  
cubic hypersurface in $\R^4$, defined by the  
vanishing of the Hankel determinant.   
It also reveals that this  
four-dimensional Carath\'eodory  orbitope is $2$-neighborly:  
the extreme points are the rank one matrices, and any two of them are  
connected by an edge. The typical intersection  
of ${\rm conv}(G \cdot v)$ with a three-dimensional affine plane looks like   
an inflated tetrahedron.  
This three-dimensional convex body is bounded by  
{\em Cayley's cubic surface}, shown in  
Figure~\ref{F:fat_tetrahedron}.  
Alternative pictures of this convex body can be found in  
\cite[Fig.~3]{Ranestad} and \cite[Fig.~4]{SU}.  
The four vertices of the tetrahedron lie on the curve $G \cdot v$, and its  
six edges are inclusion-maximal faces of ${\rm conv}(G \cdot v)$.   
\qed  
\end{example}  
  
\begin{figure}[htb]  
\[  
    \includegraphics[height=1.9in]{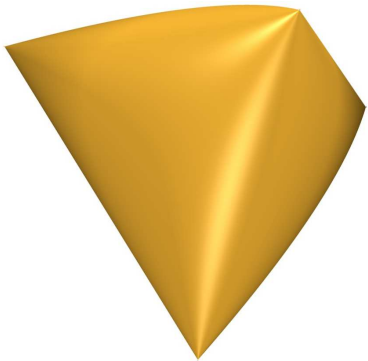}  
\]  
\vskip -0.3cm  
\label{F:fat_tetrahedron}  
\caption{Cross-section of a four-dimensional Carath\'eodory orbitope.}  
\end{figure}  
  
This article is organized as follows. We begin by deriving the  
basic definitions and a few general results about orbitopes,   
and we formulate {\bf ten key questions} which will guide  
our subsequent investigations. These are  organized along the themes of  
convex geometry (Subsection 2.1), algebraic geometry (Subsection 2.2)  
and optimization (Subsection 2.3).  
These questions are difficult. They are meant to  
motivate the reader and to guide our investigations. 
We are not yet able to offer a complete solution to any of these ten problems. 
  
Section 3 is concerned with the action of $O(n)$ and $SO(n)$  
by conjugation on $n {\times} n$-matrices.  
These decompose into the actions by conjugation on symmetric and  
skew-symmetric matrices. The resulting {\em Schur-Horn orbitopes}  
are shown be spectrahedra, their algebraic boundary is computed, and  
their face lattices are derived from certain polytopes known as  
permutahedra. The spectrahedral representations of the  
Schur-Horn orbitopes, stated in   
Theorems \ref{thm:symspec} and \ref{thm:skewsymspec},  
rely on certain Schur functors also known as  the {\em additive compound matrices}.  
  
Given a compact real algebraic group of $n {\times} n$-matrices,  
we can form its convex hull in $\R^{n \times n}$. The resulting  
convex bodies are the {\em tautological orbitopes}.  
In Subsection 4.2 we study the  
tautological orbitope and its dual coorbitope for  
the group $O(n)$. Both are spectrahedra,   
characterized by constraints on singular values,  
and they are unit balls for the operator and nuclear matrix norm  
considered in \cite{Fazel07}.  
Subsections 4.1 and 4.4 are devoted to the  
tautological orbitope for $SO(n)$.  
A characterization of its faces is given in  
Theorem \ref{thm:SOntauto}.  
  
The orbitopes of the group $SO(2)$ are the convex hulls  
of trigonometric curves, a classical topic initiated by  
Carath\'eodory in \cite{Car}, further developed  
in \cite{BN, Smi}, and our primary focus in Section 5.  
These $SO(2)$-orbitopes can be  
represented as projections of spectrahedra  
in two distinct ways: in terms of  
Hermitian Toeplitz matrices  
or in terms of Hankel matrices.  
  
A natural generalization of the rational normal curves  
in Section 5 are the Veronese varieties. Their convex  
hulls, the {\em Veronese orbitopes}, are dual to the  
cones of polynomials that are non-negative on $\mathbb{R}^n$,  
as seen in \cite{BB, Ble, Rez}. In Section 6 we undertake  
a detailed study of the   
$15$-dimensional Veronese orbitope and its  
 coorbitope which arise from  
ternary quartics.  
  
In Section 7 we mesh our investigations with a line of research  
in differential geometry.  
The {\em Grassmann orbitope} is the convex hull  
of the oriented Grassmann variety  in its Pl\"ucker embedding  
in the unit sphere in  $\wedge_d \R^n$. This vector space is 
the $d$-th exterior power of $\R^n$ and it is isomorphic to 
$\R^{\binom{n}{d}}$. 
The facial structure of Grassmann orbitopes has been  studied   
in the theory of calibrated manifolds \cite{HL1982,HarMor86, Mor85}.  
We take a fresh look at these orbitopes from the  
point of view of convex algebraic geometry.  
Theorem \ref{thm:G2n} furnishes a spectrahedral  
representation and the algebraic boundary in the special case $d=2$,  
while Theorem  \ref{thm:G36bad} shows that   
the Grassmann orbitopes fail to be spectrahedra in general.

%%%%%%%%%%%%%%%%%%%%%%%%%%%%%%%%%%%%%%%%%%%%%%%%%%%%%%%%%%%%%%%%%%%%%%%%%%%%%   
\section{Setup, Tools and Questions}  
\label{sec:Setup}  
  
Let $G$ be a real, compact, linear algebraic group, that is, a compact  
subgroup of $GL(n,\R)$ for some $n \in \N$   
given as a subvariety. Prototypic examples   
are the {\em special orthogonal group} $\,SO(n) = \{ g \in GL(n,\R) : gg^T =   
{\rm Id}_n  
\,\,{\rm and}\, \,{\rm det}(g) = 1 \}\,$ and the {\em unitary group}  
 $\,U(n) = \{ g \in GL(n,\C) : g\overline{g}^T = {\rm Id}_n  
\}$. A real representation of $G$ is a group homomorphism $\rho : G  
\rightarrow GL(V)$ for some finite dimensional real vector space $V$. We will  
write $g \cdot v := \rho(g)(v)$ for $g \in G$ and $v \in V$.  
 The representation is {\em rational} if $\rho$ is a rational map of algebraic  
varieties. By choosing an inner product $\inner{\,\cdot \,,\, \cdot\,}$ on $V$,  
we may define a $G$-invariant inner product as  
\begin{equation}\label{Eq:G-inv_form}  
    \inner{ v, w }_{G} \,\,: = \,\, \int_G \inner{g\cdot v, g\cdot w}\, d\mu.  
\end{equation}  
Here $\mu$ denotes the {\em Haar measure}, which is the unique $G$-invariant  
probability measure on $G$.   
Choosing coordinates on $V$ so that $\inner{ \cdot, \cdot }_G$ becomes the  
standard inner product on $V \simeq \R^m$, we can identify   
the matrix group $G$ with  
a subgroup of  the {\em orthogonal group} $\,O(m) =   
 \{ g \in GL(m,\R) : gg^T = {\rm Id}_m \}$.  
  
The {\em orbit} of a vector $v \in V$ under the compact group $G$ is the set  
$\,G\cdot v = \{ g \cdot v : g \in G\}$.  
This is a  bounded subset of $V$.  
The orbit $\,G \cdot v\,$ is a smooth, compact real algebraic  
variety of dimension  
\[  
    \dim G\cdot v  \,\,= \,\,\dim G - \dim \stab_G(v).  
\]  
Here $\,\stab_G(v) = \{ g \in G : g\cdot v = v \}\,$ is the {\em stabilizer} of the vector $v$. In  
particular, the orbit is isomorphic as a $G$-variety to the compact homogeneous space   
$\,G / \stab_G(v)$.  
  
The \emph{orbitope} of $G$ with respect to the vector $v \in V$ is  
the semialgebraic convex body  
\[  
    \conv(G \cdot v) \,\,= \,\, \conv \{ g \cdot v\, : \, g \in G \} \,\,\, \subset  \,\,\, V.  
\]  
We tacitly assume that the group $G$ and its representation  
$\rho$ are clear from the context and  
sometimes we write $\calO_v$ or $\calO$ for $\conv(G \cdot v)$.  
The dimension of an orbitope $\conv(G\cdot v)$ is the dimension of the affine  
hull of the orbit $G\cdot v$.    
  
For small $n$, the connected subgroups of the orthogonal group $O(n)$ are known.  
This leads to the following census of low-dimensional orbitopes of connected groups.  
  
\begin{example}[The orbitopes of dimension at most four arising from connected  
    groups]   
$\qquad$  
{\rm   
    We identify $G$ with a subgroup of $SO(n)$.  
We  assume $\calO = \conv(G \cdot  v)$ is an $n$-dimensional orbitope in $\R^n$.    
This implies that $G$ fixes no non-zero vector.  
    Here is our census:  
  
\smallskip  
  
{\sl $n=1$:} There are no one-dimensional orbitopes because  
$SO(1)$ is a point.  
  
\smallskip  
\noindent  
It is known that every proper connected subgroup of $SO(2)$ or $SO(3)$  
fixes a non-zero vector, so for $n \leq 3$, the  
subgroup $G$ must be equal to $SO(n)$.  
This establishes the next two cases:  
  
\smallskip  
  
{\sl $n=2$:} The only orbitopes in $\R^2$ are the discs  
$\{ (x,y) \in \R^2 : x^2 + y^2 \leq r^2 \}$.  
  
{\sl $n=3$:} The only orbitopes in $\R^3$ are the balls  
$\{ (x,y,z) \in \R^3 : x^2 + y^2 + z^2 \leq r^2 \}$.  
  
\smallskip  
  
\noindent The four-dimensional case is where things begin to get interesting:  
  
\smallskip  
  
{\sl $n=4$:} The group $SO(4)$ has   
connected subgroups $G$ of dimension $1$, $2$, $3$ and $6$.  
\begin{itemize}  
\item If ${\rm dim}(G) = 6$ then $G = SO(4)$ and the orbitopes are balls in $\R^4$.  
\item If ${\rm dim}(G) = 3$ then $G \simeq SU(2)$, acting as the unit quaternions  
on all quaternions, $\HH=\R^4$, by either left or right multiplication.  
Here the orbitopes are also balls in $\R^4$.  
\item If ${\rm dim}(G) = 2$ then $G \simeq SO(2)\times SO(2)$.  
 These tori act on $\R^4$ through an orthogonal direct sum decomposition   
 $\R^2\oplus \R^2$  
 and their orbitopes are products of two discs.  
\item If ${\rm dim}(G) = 1$ then $G \simeq SO(2)$ and we obtain four-dimensional  
orbitopes that are isomorphic, for some positive integers $a < b$, to the  
 Carath\'eodory orbitopes  
$$  
   \Cara_{a,b}\ :=\  \conv   
   \,\bigl\{(\cos at, \sin at, \cos bt, \sin bt) \in \R^4 \mid t\in[0,2\pi]\bigr\}\,.  
$$  
These orbitopes were introduced one century ago by  
Carath\'eodory \cite{Car}. Their study was picked up  
in the 1980s by Smilansky~\cite{Smi} and recently by  
Barvinok and Novik \cite{BN}. We note that  
$\Cara_{1,2}$ is affinely isomorphic to the Hankel orbitope in Example~\ref{ex:HankelSmall}.  \qed  
\end{itemize}}  
\end{example}  
  
To compute the dimensions of orbitopes in general we shall need a pinch of  
representation theory \cite{FH}.  A representation $V$ of the  group  
$G$ is {\em irreducible} if its only subrepresentations are $\{0\}$ and $V$.  
If $V$ and $W$ are irreducible representations, then the space $\Hom_G(V,W)$  
of equivariant linear maps between them is zero unless $V\simeq W$.  Schur's  
Lemma states that $\End_G(V):=\Hom_G(V,V)$ is a division algebra over $\R$,  
that is, either $\R$, $\C$, or $\HH$.  If $W_1,W_2,\dotsc$ is a complete list  
of distinct irreducible representations of $G$, and $V$ is any representation  
of $G$, then we have a canonical decomposition into isotypical representations:  
\begin{equation}  
\label{eq:isotypical}  
   \bigoplus_{i\geq 0} \Hom_G(W_i,V)\otimes_{\SEnd_G(W_i)} W_i\   
\xrightarrow{\  \simeq \ }\ V.  
\end{equation}  
This is an isomorphism of $G$-modules.  
The map (\ref{eq:isotypical}) on each summand is $(\varphi,w)\mapsto \varphi(w)$.  
The image of  
the $i$th summand in $V$ is called the {\em $W_i$-isotypical component  
of $V$}, and when it is non-zero, we say that  
the irreducible representation $W_i$ {\em appears} in  
the $G$-module $V$.  
We say that $V$ is {\em multiplicity-free} if each irreducible representation $W_i$
appears in $V$ at most once, so that $\Hom_G(W_i,V)$ has rank 1 or 0 over $\End_G(W_i)$. 
  
Suppose that $V$ contains the trivial representation and write $V=\R^l\oplus  
V'$, where $\R^l$ is the trivial isotypical component of $V$ and $V'$ does not  
contain the trivial representation.    
Any vector $v\in V$ can be written as  
$v=v_0\oplus v'$, where $v_0\in\R^l$ and $v'\in V'$.  Then  
\[  
    G\cdot v\ =\ v_0\oplus G\cdot v'  
    \qquad\mbox{and}\qquad  
    \conv(G\cdot v)\ =\ v_0\oplus \conv(G\cdot v')\,.  
\]     
Thus we lose  no geometric information in assuming that $V$ does not contain the trivial  
representation. This property ensures that the linear span of an orbit coincides  
with its affine span. Hence the affine span of an orbitope decomposes   
along its isotypical components:  
\[  
   \aff(G\cdot v)\ =\ \bigoplus_{i \geq 0} \aff(G\cdot v_i)\,  
      \qquad \hbox{for vectors} \,\,\,  
   v = \oplus_i v_i \,\in V\,\,\, \hbox{as in (\ref{eq:isotypical}).}  
\]  
To determine the dimension of $\aff(G\cdot v)$ for $v$ in a single isotypical  
component $V$ we proceed as follows. Let $V = W^l$ with $W$ irreducible.  
Then $v=(w_1,\dotsc,w_l)$ and the affine span of $G\cdot v$ is isomorphic to  
$W^k$, where $k$ is the rank of the $\End_G(V)$-module spanned by  
$w_1,\dotsc,w_l$. Hence, the dimension of $\aff(G \cdot v)$  
over $\R$ equals $k \cdot \dim W$.  
In particular, $k\leq \dim_{\SEnd_G(V)}(W)$. 
If $V$ is multiplicity free and $v$ has a nonzero projection into each isotypical
component of $V$, then $\dim \conv(G\cdot v) = \dim V$. 

We see this in the Carath\'eodory orbitopes for the group $SO(2)$ of $2\times 2$ rotation
matrices.
Its nontrivial representations are $W_a\simeq \R^2$, where a rotation matric acts through
its $a$th power, for $a>0$.
Let $\Cara_{a,b}$ be the orbitope of $SO(2)$ with respect to a general vector 
$v\in W_a\oplus W_b$.
If $a\neq b$, then $V$ has two isotypical components and  
$\Cara_{a,b}$ has dimension four. If $a=b$,   
then $V \simeq W_a^2$ consists of a  
single isotypical component and $\Cara_{a,a}$ is   
two-dimensional, as $\End_{SO(2)}(\R^2)=\C$, the span of $(w_1,w_2)$ is complex
one-dimensional.
  
%%%%%%%%%%%%%%%%%%%%%%%%%%%%%%%%%%%%%%%%%%%%%%%%%%%%%%%%%%%%%%%%%%%%%%  
\subsection{Convex geometry}  
  
Orbitopes are convex bodies, and it is natural to begin their  
study from the perspective of classical convexity.  
A point $p$ in a convex body $K \subset V$ is an {\em extreme point}  
if $\conv(K\backslash \{p\}) \not= K$.   
Thus, the set $E$ of extreme points of $K$ is the minimal subset satisfying  
${\rm conv}(E) = K$. An extrinsic description of $K$ is given by its   
\emph{support function}   
\[  
   h(K,\, \cdot\,) : V^* \rightarrow \R, \,  
   \,\ell\, \mapsto\,  
    h(K,\ell)\, := \,\max \{ \ell(x) : x \in K \}.   
\]  
In terms of the support function, the convex body $K$ is the set of points $x  
\in V$ such that $\ell(x) \le h(K;\ell)$ for every $\ell \in V^*$. Each  
linear functional $\ell \in V^*$ defines an {\em exposed face} of $K$:  
\[  
    K^\ell\,\,\, = \,\,\,\{ \,p \in K : \ell(p) = h(K;\ell) \, \}.  
\]  
An exposed face $K^\ell$ is itself a convex body of dimension $\dim  
\aff(K^\ell)$. An exposed face of dimension $0$ is called an \emph{exposed  
point} of $K$. It follows that every exposed point is extreme, but the  
inclusion is typically strict. However, for orbitopes, these two notions  
coincide.  
  
\begin{proposition}  
    Every point in the orbit $\,G\cdot v\,$ is exposed in   
its convex hull.  In  particular, every extreme point of the  
orbitope $\,\conv(G \cdot v)\,$ is an exposed point.  
\end{proposition}  
  
\begin{proof}  
    Since $G$ acts orthogonally on $V$, the orbit $\,G\cdot v\,$ lies   
entirely in the sphere in $V$ that is centered at $0$ and contains the point  
 $v$.  As every point of the sphere is exposed, the  
    entire orbit consists of exposed points and hence extreme points.   
\end{proof}  
  
A closed subset $F \subseteq K$ is a \emph{face} if $F$ contains  
the two endpoints to any open segment in $K$ it intersects.    
This includes $\emptyset$ and $K$ itself.  An inclusion-maximal proper face of $K$ is  
called a \emph{facet}.  By separation, every face is contained in an exposed  
face and thus facets are automatically exposed. In general, every exposed face  
is a face but not conversely.  
  
\begin{question}  
    When does an orbitope have only exposed faces?  
\end{question}  
  
The exposed faces of a convex body form a partially ordered set with respect 
to inclusion, called the \emph{face lattice}. The face lattice is atomic but 
in general not coatomic as was pointed out to us by Stephan Weis.  A necessary 
condition is that the polar body (see below) has all faces exposed 
(cf.~\cite{Weis}).  Also, it is generally not graded because ``being a face 
of'' is not a transitive relation. For example, the four-dimensional 
Barvinok-Novik orbitope in Section~\ref{S:Caratheodory} has triangular exposed 
faces for which the three vertices are exposed but the edges are not.  
Similar behavior is seen in the convex body on the right of Figure~\ref{F:trig_curves},
which has two triangular exposed facets whose edges and two of three vertices are not
exposed. 
  
\begin{question} \label{question2}  
    Describe the face lattices of orbitopes.  
\end{question}  
  
For an orbitope $\calO = \conv( G\cdot v)$, the faces come in $G$-orbits  
and these $G$-orbits come in algebraic families. In particular, the  
zero-dimensional faces come in a family parametrized by $G$.  
The description of these families is the point of Question \ref{question2}.   
For instance, the orbitope in Example~\ref{ex:HankelSmall} is a  
four-dimensional, two-neighborly convex body.   
Its exposed points are parametrized by the circle  
$\sphere^1$ and the edges come in a two-dimensional family.  
 
The {\em polar body}  
\[  
\calO^\circ \,\,= \,\, \{ \ell \in V^* : \h(\calO;\ell) \le 1 \}  
\]  
is called the \emph{coorbitope} of $G$ with respect to $v\in V$.  
Our assumption that $V$ does not contain the trivial representation  
ensures that $0$ is the centroid of $\calO$, and this implies  
$\,(\calO^\circ)^\circ = \calO$.  
We shall also make use of the cone over the coorbitope   
$\calO^\circ$. This is the \emph{coorbitope cone}  
\begin{equation}  
\label{eq:coorbitopecone}  
 \widehat{\calO^\circ} \,\, =\,\,   
 \bigl\{ (\ell,m) \in V^* \oplus \R \,: \,\h(\calO;\ell) \le m \bigr\}.  
\end{equation}  
  
For a convex body $K$ the assignment $\|x\|_K := \inf \{ \lambda >  0 :  
\lambda x \in K \}$ defines an \emph{(asymmetric) norm} on $V$ with unit ball  
$K$.  
If $K$ is centrally-symmetric with respect to the origin, then $\| \cdot \|_K$  
is an actual norm.  In that case   
the polar body $K^\circ$ is also centrally-symmetric and  
$\|\cdot\|_{K^\circ}$ is  
the {\em dual norm}.  Norms and support functions are related by  
\[  
    \|\ell\|_{K^\circ} \,\, = \,\, h(K;\ell)  
\, \quad \hbox{ for all $\ell \in V^*$.}  
\]  
 In particular, if $K$ is an orbitope, then $\| \cdot \|_K$  
and  $\|\cdot\|_{K^\circ}$ are $G$-equivariant norms.  
  
Every point $p$ in a convex body $K$ is  
in the convex hull of finitely many extreme points.  
We denote by $d_p$ the least cardinality of  a set $E$ of extreme points  
with $p \in \conv(E)$. We call $\cara(K) := \sup \{ d_p : p \in K \}$ the  
\emph{Carath\'eodory number} of $K$. Carath\'eodory's Theorem (see  
e.g.~\cite[\S I.2]{Barvinok}) asserts that $\cara(K)$ is bounded from  
above by  
$\dim K + 1$.  Fenchel showed that  
$\cara(K) \le \dim K$ when the set of extreme points of $K$ is connected  
 \cite{fenchel29}. Note that   
the Carath\'eodory number of an orbitope $\calO_v = \conv(G\cdot v)$ in general  
depends on $v$ (cf.~\cite[Theorem~4.9]{LSS08})   
whereas, for multiplicity-free representations, the dimension of $\calO_v$ does not.  
  
\begin{question}  
What are the  Carath\'eodory numbers of orbitopes?  
\end{question}  
  
We refer to the recent work of Barvinok and Blekherman \cite{BB, Ble}  
for more information about the convex geometry of  
orbitopes and coorbitopes, especially with regard to their volumes.  
   
\subsection{Algebraic geometry}  
  
Here we look at orbitopes as objects in real algebraic geometry.  
Fix a rational representation $\rho : G \rightarrow GL(m,\R)$   
of a compact connected algebraic group $G$. Every orbit $\,G \cdot v\,$   
is an irreducible real algebraic variety in $\R^m$, and we  
may ask for its prime ideal.  
By the Tarski-Seidenberg  
Theorem \cite[\S 2.4]{BPR},  
the orbitope is a semi-algebraic set.  
  
\begin{question}  
    Which orbitopes are \emph{basic} semi-algebraic sets, i.e.\   
    for which triples $(G,\rho,v)$   
    can ${\rm conv}(G \cdot v)$ be  
    described by a finite conjunction of polynomial equations and inequalities?  
\end{question}  
  
The boundary $\partial \calO$ of an orbitope   
$\calO$ in $\R^m$ is a compact semi-algebraic set of codimension~one  
in its affine span ${\rm aff}(\calO)$.  
The Zariski closure of $\partial \calO$ is denoted by    
$\partial_ a\calO$. We call it the {\em algebraic boundary} of   
$\calO$. If ${\rm aff}(\calO) = \R^m$ then  
the algebraic boundary $\partial_ a\calO$ is the   
zero set of a unique (up to scaling) reduced polynomial   
$f(x_1,\ldots,x_m)$ whose coefficients lie in the field of definition  
of $(G,\rho,v)$.    
That field of definition will often be the rational numbers  $\Q$.  
Since scalars in $\mathbb{Q}$ have an exact representation  
in computer algebra,  
but scalars in $\R$ require numerical floating point approximations,  
we seek to use $\mathbb{Q}$ instead of $\R$ wherever possible.  
  
\begin{question} \label{ques:algbound}  
  How to calculate the algebraic boundary $\,\partial_a\calO $ of an orbitope  $ \calO$?  
\end{question}

The irreducible factors of the polynomial $f(x_1,\ldots,x_m)$ that cuts out  
$\partial_a \calO$ arise from various singularities in the boundary $\partial  
\calO^o$ of the coorbitope $\calO^o$.  We believe that a complete answer to  
Question \ref{ques:algbound} will involve a Whitney stratification of the real  
algebraic hypersurface $\partial_a \calO^o$. Recall that a {\em Whitney  
stratification} is a decomposition into locally closed submanifolds (strata)  
in which the singularity type of each stratum is locally constant along the  
stratum.  The faces of a polytope form a Whitney stratification of its  
boundary, which is dual to the stratification of the polar polytope.  We  
expect a similar duality between the Whitney stratification of the boundary of  
an orbitope and of the boundary of its coorbitope.   
  
\begin{question}   
How to compute and study the algebraic boundary  
 $\,\partial_a\calO^o $ of the coorbitope  
   $ \calO^o$? Is every  
   component of $   \partial_a \calO$   
   the dual variety to a stratum in a Whitney stratification of  
   $\partial_a \calO^o$?  
   \end{question}  
     
Recall that the {\em dual variety} $X^\vee $ of a subvariety $X$ in $\R^m$  
is the Zariski closure of the set of all affine hyperplanes that are tangent to $X$  
at some regular point. When studying this duality, algebraic geometers   
usually work in complex projective space $\mathbb{P}_\mathbb{C}^m$  
rather than real affine space $\mathbb{R}^m$. In some of the examples  
for $G = SO(n)$ seen in this paper,  the algebraic boundary   
$\partial_a \calO^o$ of the coorbitope $\calO^o$ coincides with  
 the dual variety $X^\vee $ of the orbit $X = G \cdot v$.  
A good example for this is the discriminantal hypersurface  
in Corollary \ref{cor:discrorbi}.  
More generally, we have the impression that the hypersurface   
 $\partial_a \calO^o$ is often irreducible  
 while $\partial_a \calO$ tends to be reducible.  
 For further  appearances of dual varieties in convex algebraic geometry  
 see \cite{Ranestad, SU}.  
   
The \emph{$k$-th secant variety} of $G \cdot v$ is the Zariski closure of  
all $(k {+} 1)$-flats spanned by points of $G\cdot v$.   
The study of secant varieties leads to lower bounds for  the Carath\'eodory number:  
  
\begin{proposition} \label{prop:secant}  
If $k \geq \cara(\calO_v) $ then the  
$k$-th secant variety of $G \cdot v$ is the ambient space~$\R^m$.  
\end{proposition}  
  
\begin{proof}  
Let $k \geq \cara(\calO_v)$. The set of points that lie in  
the convex hull of $k+1$ points of $G \cdot v$ is dense in  
$\calO_v$ and hence is Zariski dense in $\R^m$.  
The $k$-th secant variety contains that set.  
\end{proof}  
  
The lower bound for $\cara(\calO_v)$   
from Proposition \ref{prop:secant} usually does not match  
Fenchel's upper bound $\cara(\calO_v) \leq {\rm dim}(\calO_v)$.  
For instance, consider the Carath\'eodory orbitope $\calO_v$   
 in Example \ref{ex:HankelSmall}. Its algebraic boundary  
 $\partial_a \calO_v$ equals the  
second secant variety of the orbit $G \cdot v$, so   
 Proposition \ref{prop:secant} implies  
  $\cara(\calO_v) \geq 3$. This orbitope satisfies  
  $\cara(\calO_v)= 3$ but ${\rm dim}(\calO_v) = 4$.   
  
\begin{question}   
For which orbitopes  
$\mathcal{O}$ is the boundary $\partial_a \mathcal{O}$ one of the secant varieties of $G \cdot v$~?  
When is the lower bound on   
the Carath\'eodory number $\cara(\calO)$ in Proposition \ref{prop:secant} tight ?  
\end{question}  
  
\subsection{Optimization}  
  
A fundamental object in convex optimization is the set ${\rm PSD}_n$  
of positive semidefinite symmetric real $n {\times} n$-matrices.  
This is the closed basic semi-algebraic cone defined by the  
non-negativity of the $2^n - 1$ principal minors. It can also be  
described by only $n$ polynomial inequalities, namely, by requiring that the 
elementary symmetric polynomials in the eigenvalues, i.e.\ the suitably normalized 
coefficients of the characteristic polynomial, be non-negative. 
The algebraic boundary $\partial_a {\rm PSD}_n$ of the cone ${\rm PSD}_n$ 
is the symmetric $n {\times} n$-determinant.  
All faces of ${\rm PSD}_n$ are exposed, isomorphic to ${\rm PSD}_k$ for  
$k \le n$, and indexed by the lattice of linear subspaces ordered by  
reverse inclusion.  
  
Spectrahedra inherit these favorable properties. Recall that  
a {\em spectrahedron} is the intersection of the cone ${\rm PSD}_n$ with an   
affine-linear subspace in ${\rm Sym}_2(\R^n)$. If we know that  
an orbitope is a spectrahedron then this  
 either answers or simplifies many of our questions.  
   
 \begin{question} \label{ques:IsSpec1}  
 Characterize all ${\rm SO}(n)$-orbitopes that are spectrahedra.  
 \end{question}  
  
Polytopes are special cases of spectrahedra: they arise  
when the affine-linear space consists of diagonal matrices.  
One major distinction between polytopes and spectrahedra is that  
the class of spectrahedra is not closed under projection. That is,   
the image of a  
spectrahedron under a linear map is typically not a spectrahedron.  
See Section 5 for orbitopal examples.  
Characterizing projections of spectrahedra among all  
convex bodies is a major  
open problem in optimization theory; see e.g.~\cite{HelNie}.   
Here is a special case of this general problem:  
  
\begin{question} \label{ques:IsSpec2}  
Is every orbitope the linear projection of a spectrahedron?  
\end{question}  
  
In Questions \ref{ques:IsSpec1} and \ref{ques:IsSpec2}, it  
is important to keep track of the subfield of $\R$ over which the data  
 $(G,\rho,v)$ are defined. Frequently, this subfield is the  
 rational numbers $\Q$, and in this case we seek to write  
 the orbitope as a (projected) spectrahedron over $\Q$  
 and not just over $\R$.  
  
Semidefinite programming is the problem of maximizing  
a linear function over a (projected) spectrahedron, and there are   
efficient numerical algorithms for solving this problem in practice.  
In our context of orbitopes, the optimization problem can be phrased as follows:  
  
\begin{question}  
What method can be used for maximizing a  
         linear function $\ell$ over an orbitope $\calO$?   
    Equivalently,   
    how to evaluate the norm   $\ell \mapsto \| \ell \|_{\calO^\circ}$ associated  
    with the coorbitope $\calO^\circ$ ?  
\end{question}  
  
This is equivalent to a non-linear optimization problem over the compact group $G$.  
We seek to find $g \in G$ which maximizes  
$\,\ell ( \rho(g) \cdot v)$.  
This maximum is an algebraic function of $v$.  
  
\section{Schur-Horn Orbitopes}  
\label{sec:SchurHorn}  
  
In this section we study two families of orbitopes for the  
orthogonal group $G = O(n)$.   
This group acts on the Lie algebra $\mathfrak{gl}_n$ by restricting the adjoint  
representation of $GL(n,\R)$. The   
$O(n)$-module  $\mathfrak{gl}_n $ decomposes into two  
distinguished invariant subspaces, namely $\sym$ and $\skew$.  
These  correspond to the normal and tangent space of $O(n) \subset GL(n,\R)$ at the identity.   
In matrix terms, the spaces of symmetric and skew-symmetric matrices form two  
natural representations of $O(n)$ for the action $g \cdot A = g A g^T$ with  
$g \in O(n)$ and $A \in \R^{n \times n}$.   
  
For a symmetric matrix $M \in \sym$ we define the  
\emph{symmetric Schur-Horn orbitope}  
\begin{align*}  
    \SH_M & \,:= \, \conv( G\cdot M ) \,\, \subset\, \sym . \\  
    \intertext{For a skew-symmetric matrix  
    $N \in \skew$ we define the 
     \emph{skew-symmetric Schur-Horn orbitope}}  
    \SH_N &\,:= \,\conv( G\cdot N ) \subset \skew.  
\end{align*}  
We shall see that these orbitopes are intimately related to   
 certain polytopes which govern their  boundary structure  
and  spectrahedral representation. This connection arises  
via the classical Schur-Horn theorem~\cite{schur23}.  
 The material in the sections below  
could also be presented in symplectic language,  
using the moment maps of Atiyah-Guillemin-Sternberg~\cite{At82, GS82}. 
  
\subsection{Symmetric Schur-Horn orbitopes}   
  
The $\binom{n+1}{2}$-dimensional space $\sym$  
decomposes into the trivial $O(n)$-representation, given by multiples of the identity
matrix, and the irreducible representation of symmetric   
$n {\times} n$-matrices with trace zero. 
%So, it suffices to study symmetric Schur-Horn orbitopes $\SH_M$ when $\Tr(M) = 0$. 
Every symmetric matrix  
 $M \in \sym$ is orthogonally diagonalizable over $\R$.  
  The ordered list of eigenvalues of $M $ is denoted  
 $\l(M) = (\l_1(M) \ge \l_2(M) \ge  
\cdots \ge \l_n(M))$. The orbit  
$G\cdot M$ equals the set of matrices $A \in \sym$ that satisfy  
 $\l(A) = \l(M)$. We shall see that its convex hull  
 $\SH_M = {\rm conv}(G \cdot M)$ is   
the set of matrices $A \in \sym$ for which $\l(A)$ is majorized by $\l(M)$.   
  
For $p,q \in \R^n$ we say that $q$ is \emph{majorized} by $p$,  
written $q \maj p$,  if  $\,q_1 + q_2 + \cdots + q_n \,=\,  
    p_1 + p_2 + \cdots + p_n$, and, after reordering,  
     $\,q_1 \ge \cdots \ge q_n\,$ and $\,p_1 \ge \cdots \ge p_n$, we have  
\[  
    q_1 + q_2 + \cdots + q_k \ \le \  
    p_1 + p_2 + \cdots + p_k \quad\,\, \hbox{for $\,k=1,\dots,n-1$.}   
\]  
For a fixed point $p \in \R^n$, the set of all points $q$ majorized by $p$ is   
the convex polytope   
\[  
    \Pi(p) \,\, =\,\, \{ q \in \R^n : q \maj p \} \,\, = \,\,\conv \{   
        \pi\cdot p = ( p_{\pi(1)}, \dots, p_{\pi(n)}) :   
        \pi \in \mathfrak{S}_n \}.  
\]  
Here $\mathfrak{S}_n$ denotes the symmetric group,  
and $\Pi(p)$ is the \emph{permutahedron} with respect to $p$. This  
is a well-studied polytope \cite{onn93, Ziegler} and is itself an orbitope for $\mathfrak{S}_n$.  
The permutahedron $\Pi(p)$ for  
$p = ( p_1 \ge p_2 \ge \cdots \ge p_n)$ consists of all  
points $q \in \R^n$ such that  $\sum_i p_i = \sum_i q_i$ and  
\begin{equation}  
\label{eq:permineq}  
    \sum_{i \in I} q_i \ \le \ \sum_{i=1}^{|I|} p_i  
    \qquad \hbox{for all $I \subseteq [n]$.}  
\end{equation}  
  
Let $\diag : \sym \rightarrow \R^n$ be the linear  
projection onto the diagonal.   
  
\begin{proposition}[The symmetric Schur-Horn Theorem~\cite{Leite99}]  
    Let $M \in \sym$ and $\SHsym_M$ its symmetric Schur-Horn orbitope.  
    Then the diagonal $\diag(M)$ is majorized by the vector of  
    eigenvalues $\l(M)$. In fact, the orbitope $\SHsym_M$   
    maps linearly onto the permutahedron:  
        \[        \diag(\SHsym_M) \,\,=\,\, \Pi(\l(M)).    \]  
\end{proposition}  
  
\begin{corollary} \label{cor:SH1}  
The Schur-Horn orbitope equals  
 $\, \SHsym_M = \{ A \in \sym : \l(A) \maj \l(M) \}$.  
\end{corollary}  
  
\begin{proof}  
We have shown that the right hand side equals   
$\{ A \in \sym : \l(A) \in \diag(\SHsym_M) \}$.  
We claim that a matrix $A$ is in this set if and only if  
$A$ lies in $ \SHsym_M$. This is clear if $A$ is a diagonal matrix.  
It follows for all matrices since  
both sets are invariant under the $O(n)$-action.  
\end{proof}  
  
Our next goal is to derive a spectrahedral characterization of $\SHsym_M$.  
Consider the natural action of the Lie group $GL(n,\R)$ on the $k$-th  
exterior power $\wedge_k \R^n$. If $\{v_1,v_2,\dots, v_n \}$  
is any basis of  $\R^n$, then the induced basis vectors of $\wedge_k \R^n$ are  
  $v_{i_1} \wedge v_{i_2} \wedge \cdots \wedge  
v_{i_k}$ for $1 \le i_1 < i_2 < \cdots < i_k \le n$. A matrix  
$g \in GL(n,\R)$ acts on a basis element by sending it to  
$g \cdot v_{i_1} \wedge  
v_{i_2} \wedge \cdots \wedge v_{i_k} = (g \cdot v_{i_1}) \wedge (g \cdot  
v_{i_2}) \wedge \cdots \wedge (g \cdot v_{i_k})$. We denote by $\Sfunc_k :  
\mathfrak{gl}(\R^n) \rightarrow \mathfrak{gl}(\wedge_k \R^n)$ the induced map  
of Lie algebras. The linear map $\Sfunc_k$ is defined by the rule  
\begin{equation}  
\label{eq:addcomprule}  
\Sfunc_k(B)(v_{i_1} \wedge v_{i_2} \wedge \cdots \wedge v_{i_k}) \,\,= \,\,  
        \sum_{j = 1}^k v_{i_1} \wedge \cdots \wedge v_{i_{j-1}} \wedge (B  
        v_{i_j}) \wedge v_{i_{j+1}} \wedge \cdots \wedge v_{i_k}.  
\end{equation}  
The $\binom{n}{k} \times \binom{n}{k}$-matrix that represents  
$\Sfunc_k(B)$ in the standard basis of $\wedge_k \R^n$ is known as  
the \emph{$k$-th additive  
compound matrix} of the $n \times n$-matrix $B$. It has the following main property:  
  
\begin{lemma} \label{lem:compoundeigen}  
    Let $B \in \sym$ with eigenvalues $\l(B) = (\l_1, \l_2,\dots,\l_n)$. Then  
    $\Sfunc_k(B)$ is symmetric and has eigenvalues $ \,\l_{i_1} + \l_{i_2} + \cdots  
    + \l_{i_k}\,$ for $1 \le i_1 < i_2 < \cdots < i_k \le n$.   
\end{lemma}  
\begin{proof}  
    Let $v_1,\ldots,v_n$ be an   
eigenbasis for $B$.  Then the formula (\ref{eq:addcomprule}) says  
$$  
        \Sfunc_k(B)(v_{i_1} \wedge v_{i_2} \wedge \cdots \wedge v_{i_k}) \,\, = \,\,  
        \sum_{j = 1}^k v_{i_1} \wedge \cdots \wedge v_{i_{j-1}} \wedge ( \lambda_{i_j}  
        v_{i_j}) \wedge v_{i_{j+1}} \wedge \cdots \wedge v_{i_k}.   $$  
    Hence  $\, v_{i_1} \wedge v_{i_2} \wedge \cdots \wedge v_{i_k}\,$   
    is an eigenvector of $\Sfunc_k(B)$ with eigenvalue  
        $\l_{i_1} + \cdots + \l_{i_k}$.  
         \end{proof}  
  
This leads to the result that each symmetric  
Schur-Horn orbitope $\SHsym_M$ is a spectrahedron.  
  
\begin{theorem} \label{thm:symspec}  
    Let $M \in \sym$ with ordered eigenvalues $\l(M) =  
    (\l_1 \ge \cdots \ge \l_n)$.  Then  
        \[    \SHsym_M \, = \,    \bigl\{  A \in \sym \,:\,  
    \Tr(A) = \Tr(M) \,\,\hbox{and} \,\,\,  
     \sum_{i=1}^k \l_i {\rm Id}_{\binom{n}{k}} - \Sfunc_k(A)\, \succeq \,0\,\,\,  
    \hbox{for}\,\, k = 1,\ldots,n-1 \bigr\}.  
    \]  
\end{theorem}  
  
\begin{proof}  
    A matrix $A \in \sym$ is in $\SHsym_M$ if and only if $\l(A)$ is in the  
    permutahedron    $\Pi(\l(M))$. From the inequality representation  
    in (\ref{eq:permineq}),   
    in conjunction with Lemma \ref{lem:compoundeigen},  
       we see  that this is the  
    case if and only if the largest eigenvalue of   
    $\l( \Sfunc_k(A) )$  is at most $\l_1 + \dots + \l_k$.  
\end{proof}  
  
We shall now derive the description of all faces of the Schur-Horn  
orbitope $\SHsym_M$.  Since $\SHsym_M$ is a spectrahedron,  
by Theorem \ref{thm:symspec},  
we know that all of its faces are exposed faces. Hence  
 a face of $\SHsym_M$ is the set of points maximizing  
a linear function $\ell : \sym \rightarrow \R$. The canonical $O(n)$-invariant  
inner product on $\sym$ is given by $\inner{A,B} = \Tr(AB)$ and, by identifying  
spaces, a linear function may be written as $\ell(\,\cdot \,) = \inner{B,\, \cdot \,}$.  
Note that the dual space $(\sym)^*$ is equipped with the contragredient action, that is, $g  
\cdot \ell = \inner{g^T B g, \,\cdot \,}$.  
  
\begin{theorem} \label{thm:facelatticeSH}  
    Every $O(n)$-orbit of faces of $\SHsym_M$ intersects the pullback of a  
    unique $\mathfrak{S}_n$-orbit of faces of the permutahedron $\Pi(\l(M))$.  In particular,  
    the faces of $\SHsym_M$ are products of  symmetric Schur-Horn orbitopes of  
     smaller dimensions corresponding to flags in $\R^n$.  
\end{theorem}  
  
\begin{proof}  
    Let $F$ be a face of $\SHsym_M$ and let $\ell = \inner{B, \,\cdot\,}$ be a  
    linear function maximized at $F$. Then the face $g \cdot F$ is  
    given by $g \cdot \ell$ and we may identify $O(n) \cdot F$ with the orbit  
    $O(n) \cdot \ell$. Since $B$ is orthogonally diagonalizable, the orbit  
    $O(n) \cdot \ell$ contains the diagonal matrices $\pi \cdot \l(B)$ for  
    $\pi \in \mathfrak{S}_n$. The corresponding faces are the  
    pullbacks of the orbit of faces of $\Pi(\l(M))$ given by the linear  
    function $\inner{\l(B),\cdot}_{\R^n}$.   
    Faces of the permutahedron correspond to flags of coordinate subspaces,  
    and $O(n)$ translates these to arbitrary flags of subspaces.  
\end{proof}  
  
Facets of the permutahedron $\Pi(p)$ correspond to coordinate subspaces. Replacing   
these with arbitrary subspaces yields supporting hyperplanes for the Schur-Horn  
orbitope $\SHsym_M$.  
  
\begin{corollary} \label{cor:SH2}  
    The Schur-Horn orbitope $\SH_M$ is the set of matrices $A \in \sym $ such that  
    \[  
        \Tr(A|_L) \le \Tr(M|_L) \quad  
    \hbox{for every subspace $L \subseteq \R^n$.}   
    \]  
\end{corollary}  
  
The following three examples show  
how Theorem \ref{thm:facelatticeSH}  
translates into explicit face lattices.

\begin{example}[The free spectrahedron] \label{ex:freespec}  
    Let $M = e_1 e_1^T \in \sym$ be the diagonal matrix with diagonal  
    $(1,0,\ldots,0)$. The  orbitope  
    $ \SHsym_M$ is the convex hull of all symmetric rank $1$  
    matrices with trace $1$, and hence  
    $\,   \SHsym_M =  {\rm PSD}_n \,\cap \,\{{\rm trace} = 1\}$.  
    This orbitope plays the role of a  
    ``simplex among spectrahedra'' because  
    every compact spectrahedron is an affine section.  
      
    The face $\SHsym_M^B$ of the orbitope  
    $\SHsym_M$ in direction $B \in \sym$ is isomorphic to   
    \[  
    \conv \{ uu^T : u \in \mathbb{S}^{n-1} \cap {\rm Eig}_{\rm max}(B) \}  
    \]  
    where $\mathbb{S}^{n-1}$ is the unit sphere and ${\rm Eig}_{\rm max}(B)$ is the eigenspace   
    of $B$ with maximal eigenvalue.  
        Thus, $\SH_M^B$ is isomorphic to a lower  
    dimensional Schur-Horn orbitope for a rank one matrix.

    We conclude that the face lattice of $\SHsym_M$ consists of the linear  
        subspaces of $\R^n$ ordered by inclusion. This fact is well known; see  
    \cite[\S II.12]{Barvinok}.  
     The dimension of a face  
    corresponding to a $k$-subspace is $\tbinom{k+1}{2}-1$.   
        The projection $\diag(\SHsym_M)$ is the standard $(n-1)$-dimensional  
    simplex $\Delta_{n-1} = \conv\{ e_1,e_2,\ldots,e_n  \}$ whose faces  
    correspond to the     coordinate         subspaces. \qed  
\end{example}  
  
\begin{example}[Spectrahedral hypersimplices]  
    We now describe continuous analogs to  hypersimplices, extending the  
    simplices in Example \ref{ex:freespec}.  Fix $0 < k < n$ and let $M \in  
    \sym$ be the diagonal matrix with $k$ ones and $n-k$ zeros. The orbitope $  
    \SH_M$ of the {\em  $(n,k)$-spectrahedral hypersimplex}, as its diagonal  
    projection $\diag( \SH_M) = \Delta(n,k)$ is the classical  
    $(n,k)$-hypersimplex; cf.~\cite[Example 0.11]{Ziegler}.  For instance, if  
    $n=4$ and $k = 2$ then $\SH_M$ is nine-dimensional and   
    $\diag(\SH_M)$ is an octahedron.  Up to $\mathfrak{S}_4$-symmetry, the  
    octahedron has one orbit of vertices and edges but two orbits of  
    triangles. The pullback of any edge is a circle, and the pullbacks  
    of the triangles are five-dimensional symmetric Schur-Horn orbitopes  
    $\calO_M$ for $\l(M) = (1,0,0)$ and $\l(M) = (1,1,0)$. Both facets are  
    isomorphic to free spectrahedra.  
    \qed  
\end{example}  
  
\begin{example}[The generic symmetric Schur-Horn orbitope]  
    Let $M \in \sym$ be a symmetric  
matrix with distinct eigenvalues, e.g. $\lambda(M)  
    =(1,2,\dots,n)$. The image of $\SHsym_M$ under the diagonal map is the  
    classical permutahedron $\Pi_n = \Pi(1,2,3,\dots,n)$. Its face lattice may  
    be described as the collection of all flags of coordinate subspaces in  
    $\R^n$ ordered by refinement.   
      
    We may associate to every $B \in \sym$ the complete flag whose $k$-th  
    subspace is  the direct sum of the eigenspaces of the first $k$ largest  
    eigenvalues.  Thus, $O(n) \cdot M$ may be identified with the complete  
    flag variety over $\R$. As for the facial structure,   
    the face $\SH_M^B$ is isomorphic to the convex hull of the orbit  
    ${\rm stab}_{O(n)}(B) \cdot M$.  Here, the stabilizer decomposes into a  
    product of groups isomorphic to $O(d_i)$ where $d_i$ is the dimension of  
    the $i$-th eigenspace of $B$.  Hence, the face $\SH_M^B$ is isomorphic to  
    a Cartesian product of generic Schur-Horn orbitopes and is  of  dimension  
    $\sum_i\tbinom{d_i+1}{2}$. The face only depends on the flag associated to  
    $B$. This implies that the face lattice of $\SH_M$ is isomorphic to the  
    set of partial flags ordered by refinement.  Again, in every orbit of  
    flags there   
%  
%   Frank: Typically, there are many such coordinate flags in an orbit  
%  
%is one   
    is a  
    flag consisting only of coordinate subspaces.  These  
    special flags form the face lattice of the standard permutahedron  
    $\Pi(1,2,\dots,d) = \diag(\SH_M)$.   
We regard $\SH_M$ as a continuous analog of the permutahedron.  
\qed  
\end{example}  
  
We conclude this subsection with a discussion of the algebraic boundary  
$\partial_a \SHsym_M$ of the Schur-Horn orbitope.  Let $\K$ be the smallest  
subfield of $\R$ that contains the eigenvalues $\lambda_1,\ldots,\lambda_n$,  
and suppose that the $\lambda_i$ are sufficiently general.  Then the  
hypersurface $\partial_a \SHsym_M$ is defined 
in the affine space $\{A\in \Sym_2(\R^n)\mid\Tr(A)=\Tr(M)\}$
by the following polynomial of  
degree $2^n-2$ in $\binom{n+1}{2}$ unknowns over the field $\K$:  
\[  
    f(A) \quad = \quad \prod_{k=1}^{n-1}{\rm det} \Bigl(\Sfunc_k(A) -  
    \sum_{i=1}^k \lambda_i \cdot {\rm Id}_{\binom{n}{k}} \Bigr) .  
\]  
However, from a computer algebra perspective, this is  not what we want.  
Assuming that $M$ has entries in $\Q$, we prefer not to pass to the field  
extension $\K$, but we want the algebraic boundary $\partial_a \SHsym_M$ to be  
the $\Q$-Zariski closure of the above hypersurface $\{f(A) = 0\}$.  For  
instance, suppose that the characteristic polynomial of $M$ is irreducible  
over $\Q$. Then we must take the product of $f(A)$ over all permutations of  
the eigenvalues $\lambda_1,\ldots,\lambda_n$, and the polynomial $g(A)$ that  
defines $\partial_a \SHsym_M$  over $\Q$ is the reduced part of that  
product. It equals  
\[  
    g(A) \quad = \quad \prod_{k=1}^{\lceil n/2 \rceil}{\rm det}  
    \bigl(\Sfunc_k(A) \oplus \Sfunc_k(-M)\bigr) ,   
\]  
where $\oplus$ denotes the {\em tensor sum} of two square matrices of the same  
size (see e.g.~\cite[\S 3]{NieStu}).  
Here the product goes only up to $\lceil n/2 \rceil$  
because the matrices $A$ and $M$  
have the same trace.  
  
  For special matrices $M$, the  
characteristic polynomial may factor over $\Q$, and in this case  the  
algebraic boundary $\partial_a \SHsym_M$ is cut out by a factor of the  
polynomial $f(A)$ or $g(A)$.  
   
\subsection{Skew-symmetric Schur-Horn orbitopes}  
  
The  space $\skew$ consists of skew-symmetric $n {\times} n$-matrices $N$.  
The eigenvalues of $N$ are purely imaginary, say $\pm  
i\hl_1,\dots,\pm i\hl_{k}$, where $i=\sqrt{-1}$, with $k = \lfloor  
\tfrac{n}{2} \rfloor$ and an additional $0$ eigenvalue if $n$ is odd.  
Thus $N$ is not diagonalizable over $\R$, but the adjoint  
 $O(n)$-action brings the matrix $N$ into the  normal form  
\[  
gNg^T \,=\,  
    \left(  
    \begin{array}{rr}  
                   & \Lambda \\  
        -\Lambda &   
    \end{array}  
    \right)  
    \text{ for $n$ even} \quad \text{ and } \quad  
    gNg^T \,=\,  
    \left(  
    \begin{array}{rcr}  
                   &   &  \Lambda \\  
              & 0 &   \\  
        -\Lambda &   &   
    \end{array}  
    \right)  
    \text{ for $n$ odd.}  
\]  
Here $g$ is a suitable matrix in $ O(n)$,  
$\Lambda$ is the diagonal matrix with diagonal $\hl_1 \ge \cdots \ge \hl_k  
\ge 0$, and we denote $\hl(N) = ( \hl_1, \hl_2, \dots, \hl_k)$.   
Let $\Sdiag : \skew \rightarrow \R^k$ be the linear map such that  
\begin{eqnarray*}  
    \Sdiag(N) &=& (N_{1,k\phantom{+1}}, N_{2,k+1},\dots,N_{k,n}) \text{\qquad if $n = 2k$,  
    and}\\  
    \Sdiag(N) &=& (N_{1,k+1}, N_{2,k+2},\dots,N_{k,n}) \text{\qquad if $n = 2k+1$.}  
\end{eqnarray*}  
  
If $N$ is in normal form as above, then  $\Sdiag(N) = \diag(\Lambda) =  
\hl(N)$.  We call $\Sdiag(N)$ the \emph{skew-diagonal} of $N$.  
In analogy to the symmetric case, the set $\Sdiag(\SH_N)$  
of all skew-diagonals arising from $\SH_M$ is   
 nicely behaved; in fact, the necessary  
changes are rather modest.   
  
For a point $q \in \R^k$ we denote by $|q| =  
(|q_1|,\dots,|q_k|)$ the vector of absolute values.  
 For  $p \in \R^k$  
let $\BPi(p)$ be the set of points $q \in \R^k$ such that $|q|$  
is {\em weakly majorized} by $|p|$. This means that  
 $|p|$ and $|q|$ satisfy the majorization  
conditions except that $\sum_i |q_i| \le \sum_i |p_i|$ is allowed.  
The polytope $\BPi(P)$ is the \emph{$B_k$-permutahedron}.  
 It is the convex  
hull of the orbit of $p$ under the action of the Coxeter group $B_k$, the  
group of all $2^k \cdot k!$  signed permutations.  
The $B_k$-permutahedron $\BPi(p)$ for   
$p = (p_1 \ge p_2 \ge \cdots \ge p_k)$ consists of all points $q \in \R^k$  
with  
\begin{equation}  
\label{eq:Bpermineq}  
    \sum_{i \in I} q_i - \sum_{j \in J} q_j \ \le \  
    \sum_{i = 1}^{|I \cup J|} p_i  
    \quad \hbox{ for any $I,J \subseteq [k]$ with $I \cap J = \emptyset$.}  
\end{equation}  
  
As expected, we have the following analog of the symmetric Schur-Horn theorem.  
  
\begin{proposition}[The skew-symmetric Schur-Horn theorem~\cite{Leite99}]  
    Let $N \in \skew$ with $\hl(N) =   
    (\hl_1 \ge \cdots \ge  
    \hl_k)$ and $\SH_N$ the skew-symmetric Schur-Horn orbitope of  
     $N$. Then  $|\Sdiag(N)|$ is weakly majorized by $(\hl_1,\dots,\hl_k)$. In  
    particular, we have $\Sdiag(\SH_N) = \BPi(\hl)$.  
\end{proposition}  
  
The same arguments as in the symmetric case yield the following results.  
  
\begin{theorem} \label{thm:facelatticeSkewSH}  
    Every $O(n)$ orbit of faces of the skew-symmetric Schur-Horn orbitope  
    $\SH_N$ contains the pullback of a unique  
    $B_k$-orbit of faces of the $B_k$-permutahedron $\BPi(\hl(N))$.  
\end{theorem}  
  
\begin{corollary}  
The  skew-symmetric Schur-Horn orbitope  
$\SH_N$ coincides with the set of skew-symmetric matrices $A$ such that  
$|\Sdiag(A)|$ is weakly majorized by  $|\Sdiag(N)|$.  
\end{corollary}  
  
\begin{example}  
Fix $n=6$ and $k=3$, and $p=(1,2,3)$. Then the system   
(\ref{eq:Bpermineq}) consists of $26$ linear inequalities,  
namely, six inequalities $\pm q_i \leq 3$,  
twelve inequalities $\pm q_i \pm q_j \leq 5$, and  
eight  inequalities $\pm q_i \pm q_j \pm q_k \leq 6$.  
Their solution set is the $B_3$-permutahedron $\BPi(1,2,3)$, 
commonly known as the {\em truncated cuboctahedron},  
and it  has $48$ vertices, $72$ edges and  
$26$ facets (six octagons, twelve squares and eight hexagons).  
A picture is shown in  Figure \ref{F:B3permuta}.

\begin{figure}[htb]  
\[  
    \includegraphics[width=6.5cm]{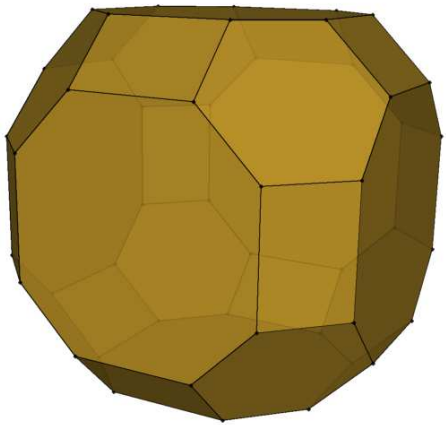}  
\]  
\vskip -0.3cm  
\caption{The $B_3$-permutahedron is the truncated cuboctahedron.}  
\label{F:B3permuta}  
\end{figure}  
  
Let $N$ denote a general skew-symmetric $6 {\times} 6$-matrix.  
Theorem \ref{thm:facelatticeSkewSH} implies that the  
facets of the $15$-dimensional orbitope $\SH_N$ come  
in three families, corresponding to $O(6)$-orbits of the octagons, squares, and hexagons 
in Figure \ref{F:B3permuta}.   
The facets arising from the octagons are skew-symmetric Schur-Horn  
orbitopes for $SO(4)$ with skew-diagonal $(1,2)$ and therefore have dimension six.  
The facets arising from the squares are the product of a line segment and a disc, 
coming from $O(2)\times SO(2)$ with $O(2)$ acting by the determinant. 
The facets arising from the hexagons are $O(3)$-orbitopes  
isomorphic to symmetric Schur-Horn orbitopes with eigenvalues $(1,2,3)$ and 
therefore have  dimension five.   
\qed  
\end{example}  
  
We next present a spectrahedral description of an  
arbitrary skew-symmetric Schur-Horn orbitope $\SH_N$.  
To derive this, we return to symmetric matrices  
and their real eigenvalues.  
  
\begin{lemma}  
    Let $N \in \skew$ be a matrix with eigenvalues $\pm i \hl_1,\dots, \pm i  
    \hl_k$ and let  
    \[  
        \hat{N}  \,\,= \,\,         \left(  
        \begin{array}{rr}  
            0 & N \\  
            -N & 0 \\  
        \end{array} \right) \,\,\in\,\,  
        \mathrm{Sym}_2\R^{2n}.  
    \]  
    Then $\hat{N}$ has eigenvalues   
    $\hl_1,\hl_1,\hl_2,\hl_2,\dots,\hl_k,\hl_k,  
    -\hl_1,-\hl_1,-\hl_2,-\hl_2,\dots,-\hl_k,-\hl_k$.  
For any $1 \leq j \leq k$, the additive compound matrix  
$ \Sfunc_{2j}(\hat{N}) $  
 has largest eigenvalue $2(\hl_1 + \hl_2 + \cdots  
+ \hl_j)$.   
    \end{lemma}

We conclude that each skew-symmetric Skew-Horn orbitope  
 $\BPi(\hl(N))$ is a spectrahedron:  
   
\begin{theorem} \label{thm:skewsymspec}  
 Let $N \in \skew$ with $\hl(N) =     (\hl_1 \ge \cdots \ge  \hl_k)$. Then  
         \[  
    \SH_N \, = \, \left\{ A \in \skew:  
    2(\hl_1 + \cdots + \hl_j) {\rm Id}_{\binom{2n}{2j}} -  
    \Sfunc_{2j}(\hat{A}) \,  
        \succeq \,0\,\,\,  
    \hbox{for}\,\, j = 1,\ldots,k \right\}.  
    \]  
\end{theorem}  
  
From this theorem we can derive a description of the  
algebraic boundary as before, and again the issue arises  
that the $\tilde \lambda_j$ lie in an extension  
$\K$ over the field of definition of $N$,  
which will usually be $\Q$.  
 At present we do not know whether  
Theorems \ref{thm:symspec} and \ref{thm:skewsymspec}  
can be extended to obtain spectrahedral representations of  
the respective orbitopes over $\Q$.  
  
We close this section with one more example of a skew-symmetric  
Schur-Horn orbitope.  
  
\begin{example}  
    Consider the skew-symmetric Schur-Horn orbitope $\SH_N$ for some $N \in  
    \skew$ with $\hl(N) = (1,0,\dots,0) \in \R^k$. According to   
Theorem \ref{thm:skewsymspec}, a spectrahedral representation~is  
    \[\SH_N \, = \, \bigl\{ \,A \in \skew:  
        \mathcal{L}_2(\hat{A}) \,\preceq \, 2 \cdot {\rm Id}_{2n} \bigr\}.    \]  
  
    The projection $\Sdiag(\SH_N)$ to the skew diagonal is the  
    \emph{crosspolytope} $\conv \{ \pm e_1, \dots, \pm e_k \}$.  
    This a regular polytope with symmetry  
    group $B_k$, and it has only one orbit of faces in each dimension.  
The orbitope $\SH_N$ is the $d=2$ instance of the {\em Grassmann orbitope}  
$\mathcal{G}_{d,n}$.  
These are important in the theory of calibrated manifolds,   
and we shall study them in  Section 7. \qed  
\end{example}

%%%%%%%%%%%%%%%%%%%%%%%%%%%%%%%%%%%%%%%%%%%%%%%%%%%%%%%%%%%%%%%%%%%%%%%%%%%%%   
  
\section{Tautological Orbitopes}  
\label{sec:tautological}  
  
We argued in Section~\ref{sec:Setup} that, given a compact group $G$ acting  
algebraically on $V \simeq\R^n$, we can identify $G$ with a subgroup of $O(n)$,  
or even  of $SO(n)$ when $G$ is connected.  
The ambient space $ \End(V) \simeq \mathfrak{gl}_n $ is itself  
an $n^2$-dimensional real  
representation of the group $G$.  
The action of $G$ on $\End(V)$ is by left multiplication.  
The orbit of the identity matrix ${\rm Id}_n$ under this action  
is the group $G$ itself. We  call the corresponding orbitope  
$\,\conv(G) =\conv(G \cdot {\rm Id}_n)\,$  
the \emph{tautological orbitope} for the pair  
$(G,V)$. This orbitope lives in $ \End(V)$, and it serves as  
an initial object because it maps linearly to every  
orbitope $\conv(G \cdot v)$ in $V$.  Tautological  
orbitopes of finite permutation groups have been studied  
 under the name of \emph{permutation polytopes}  
(see \cite{onn93}).   
The most famous of them all is the {\em Birkhoff polytope} for  
 $G = \mathfrak{S}_n$, which was studied for other Coxeter groups by 
McCarthy, Ogilvie, Zobin, and Zobin~\cite{Zobin}.
  
In this section we investigate the tautological orbitopes for  
the full groups $O(n)$ and $SO(n)$.  Similar to the Schur-Horn orbitopes in  
Section~\ref{sec:SchurHorn}, the facial structure is governed by polytopes  
arising from the projection onto the diagonal.  We begin with the  
example  $G = SO(3)$.  
  
\subsection{Rotations in $3$-dimensional space}  
  
The group $SO(3)$ of  $3 {\times} 3$ rotation matrices has dimension three. 
%is a real algebraic variety of dimension three.  
Its  tautological orbitope is a convex body of dimension nine.  
The following spectrahedral representation was suggested to us by  
 Pablo Parrilo.  
  
\begin{proposition}  
\label{prop:pabloSO3}  
 The tautological orbitope  
 ${\rm conv}(SO(3))$  is a spectrahedron  
whose boundary is a quartic hypersurface.  In fact, a  
  $3 {\times} 3$-matrix $X = (x_{ij})$ lies in  
 ${\rm conv}(SO(3))$ if and only~if  
\begin{equation}  
\label{eq:magic}  
\begin{pmatrix}  
1 {+} x_{11} {+} x_{22} {+} x_{33} & x_{32} - x_{23}    
& x_{13} - x_{31}  &  x_{21} - x_{12} \\  
x_{32} - x_{23}  & 1{+}x_{11} {-} x_{22} {-} x_{33} &    
x_{21} + x_{12}  &  x_{13} + x_{31} \\  
x_{13} - x_{31}  & x_{21} + x_{12}    
& 1 {-} x_{11} {+} x_{22} {-} x_{33} & x_{32} + x_{23} \\  
x_{21} - x_{12}  & x_{13} + x_{31}   
 & x_{32} + x_{23}  & 1 {-} x_{11} {-} x_{22} {+} x_{33}   
\end{pmatrix} \,\,\succeq \,\, 0.  
\end{equation}  
\end{proposition}

\begin{proof}   
We first claim that ${\rm conv}(SO(3))$ coincides with the set of all   
$3 {\times} 3$-matrices  
\begin{equation}  
\label{eq:pabloSU2}  
\begin{pmatrix}  
u_{11} {+} u_{22} {-} u_{33} {-} u_{44} &  2 u_{23} -2u_{14} & 2u_{13} + 2 u_{24} \\  
2 u_{23}+2 u_{14} & u_{11}{-}u_{22}{+} u_{33}{-} u_{44} & 2 u_{34}-2 u_{12} \\  
2 u_{24}-2 u_{13} &  2 u_{12}+2 u_{34} & u_{11} {-} u_{22}{-} u_{33} {+} u_{44}  
\end{pmatrix}  
\end{equation}  
where $U = (u_{ij})$ runs over all  
positive semidefinite  $4 {\times} 4$-matrices having trace $1$.

Positive semidefinite $4 {\times} 4$-matrices with  
both trace $1$ and rank $1$ are of the form  
$$  
U \,\,\, = \,\,\,  
\frac{1}{a^2+b^2+c^2+d^2}  
\begin{pmatrix}  
a^2 & ab & ac & ad \\  
ab & b^2  & bc & bd \\  
ac & bc & c^2 & cd \\  
ad  & bd & cd & d^2  
\end{pmatrix}.  
$$  
Their convex hull is the free spectrahedron of Example \ref{ex:freespec}.  
The image of the above rank $1$ matrices $U$ under the linear map  
(\ref{eq:pabloSU2}) is precisely the group  $SO(3)$.  
This parametrization is known as the {\em Cayley transform}.  
Geometrically, it corresponds to the double cover  
$SU(2) \rightarrow SO(3)$.   
The claim follows because the linear map  
commutes with taking the convex hull.  
  
The symmetric $4 {\times} 4$-matrices $U = (u_{ij})$ with  
${\rm trace}(U) = 1$ form a nine-dimensional affine space,  
and this space is isomorphic to the nine-dimensional  
space of all $3 {\times} 3$-matrices $X = (x_{ij})$ under  
the linear map given in \eqref{eq:pabloSU2}. We can express  
each $u_{ij}$ in terms of the $x_{kl}$ by  
inverting the linear relations  
$ x_{11} = u_{11} + u_{22} - u_{33} - u_{44}$,   
$ x_{12} = 2 u_{23} - 2u_{14}$, etc.  
The resulting symmetric $4 {\times} 4$-matrix $U $  
is precisely the matrix \eqref{eq:magic}  in the statement of  
Proposition \ref{prop:pabloSO3}.  
\end{proof}

The ideal of the group $O(3)$ is generated by  
the entries of the $3 {\times} 3$-matrix $X \cdot X^T - {\rm Id}_3$,  
while the prime ideal of $SO(3)$ is that same ideal plus  
$\langle {\rm det}(X)-1 \rangle$. We can check that the  
prime ideal of $SO(3)$ coincides with  
the ideal generated by the $2 {\times} 2$-minors  
of the matrix~\eqref{eq:magic}.   
Thus the group $SO(3)$ is recovered as the set of   
matrices~\eqref{eq:magic} of rank one.  
  
Proposition \ref{prop:pabloSO3} implies that ${\rm conv}(SO(3))$  
is affinely isomorphic to the free spectrahedron for $n=4$,  
that is, to the set of positive semidefinite $4 {\times} 4$-matrices  
with trace equal to $1$. This implies a characterization of  
all faces of the tautological orbitope for $SO(3)$.  
First, all faces are exposed because ${\rm conv}(SO(3))$  
is a spectrahedron. All of its proper faces are free spectrahedra,  
for $n=1,2,3$, so they have dimensions $0$, $2$ and $5$,  
as seen in Example \ref{ex:freespec}.

\subsection{The orthogonal group}  
We now examine   
$\,O(n) \, = \,\{X \in \R^{n \times n} : X \cdot X^T = {\rm Id}_n\}$.  
As before, let $\diag : \R^{n \times n} \rightarrow \R^n$   
denote the projection of the $n {\times} n$-matrices onto  
their diagonals.  
  
\begin{lemma} \label{lem:Ontodiagonal}  
The projection $\diag({\rm conv}(O(n))$ of the tautological  
    orbitope for the orthogonal group $O(n)$ to its   
diagonal is precisely the $n$-dimensional cube $[-1, +1]^n$.  
\end{lemma}  
  
\begin{proof}  
  The columns of a matrix $X \in O(n)$ are %pairwise orthogonal   
  unit vectors. 
  Thus every coordinate $x_{ij}$ in bounded by $1$ in absolute  
  value and $\diag({\rm conv}(O(n))$ is a subset of the cube.  
  For the reverse inclusion, note that  
  %The reverse  inclusion follows by observing that  
  all $2^n$ diagonal matrices with entries $\pm1$ are orthogonal matrices.  
\end{proof}  
  
The cube $[-1,+1]^n$ is the special   
$B_n$-permutahedron $\BPi(1,1,\dots,1)$.    
As with   
Schur-Horn orbitopes, the projection onto  this polytope reveals the facial  
structure. As general endomorphisms are not normal, the key concept of  
diagonalizability is replaced by that of \emph{singular value decomposition}.  
Recall that for any linear map $A \in \R^{n \times n}$ there are  
orthogonal transformations $U, V \in O(n)$ such that $U A V^T$ is diagonal  
with entries $\s(A) = (\s_1(A) {\ge} \s_2(A) {\ge} \cdots {\ge} \s_n(A)) \in  
\R^n_{\ge0}$. These entries are the \emph{singular values} of  
$A$.   
  
We shall see in (\ref{eq:Onspec}) that ${\rm conv}(O(n))$ is  
a spectrahedron, hence all of its faces are exposed faces.  
The following result recursively characterizes all faces of this  
tautological orbitope.

\begin{theorem} \label{thm:tautoOn}  
Let $\ell =  \inner{ B, \cdot \,}$ be a linear function on $\R^{n \times n}$ with   
$B \in \R^{n \times n}$.   
Then the face of ${\rm conv}(O(n))$  in direction   
$\ell$ is isomorphic to ${\rm conv}(O(m))$ where $m = \dim \ker (B)$.   
\end{theorem}  
  
\begin{proof}  
    Let $\ell( \cdot ) = \inner{ B, \cdot \,}$ be a linear function with $B \in  
\R^{n \times n}$, so that $\ell(A)=\mbox{Trace}(AB)$.   
We fix a singular value decomposition $U\Sigma V = B$   
of the matrix $B$.   
Here $\Sigma$ is a diagonal matrix $n {\times} n$ with its first   
$\,n-m \,$ entries positive and remaining $m$ entries zero.  
This matrix also defines a linear function   
 $\ell'( \cdot ) = \inner{ \Sigma, \cdot \,}$ on $\R^{n \times n}$.  
Cyclic invariance of the trace ensures that the faces  
${\rm conv}(O(n))^\ell$ and  
${\rm conv}(O(n))^{\ell'}$ are isomorphic.  
The subset of $O(n)$ at which $\ell'$ is maximized is  
the subgroup $\{ g \in O(n) : g \cdot e_i = e_i \text{ for } i = 1,\dots,n-m  
\}$.  
The convex hull of this subgroup equals  
 ${\rm conv}(O(n))^{\ell'}$. It  
coincides with the tautological orbitope for $O(m)$.  
\end{proof}  
  
We interpret Theorem \ref{thm:tautoOn} geometrically as  
saying that the tautological orbitope for $O(n)$  
is a continuous analog of the $n$-dimensional cube. Every face of the  
cube is a smaller dimensional cube and the dimension of a face maximizing a  
linear functional $\ell$ is determined by the support of $\ell$.  
The role of the support is now played by the rank of the matrix $B$. This  
behavior yields information about  
 the Carath\'{e}odory number of the tautological orbitope.  
  
\begin{proposition}  
    The Carath\'{e}odory number of the orbitope $\conv( O(n) )$  
    is at most $n+1$.  
\end{proposition}  
\begin{proof}  
    By  \cite[Lemma~3.2]{LSS08}, the Carath\'{e}odory number of a  
    convex body $K$ is bounded via  
    \[  
        \cara(K)\,\, \le \,\,  
1 + \max\{ \cara(F) : F \subset K \text{ a proper face }  
        \}.  
    \]  
    Since every proper face is isomorphic to $\conv(O(k))$ for some  
    $k < n$ the result follows by induction on $n$. The base case is $n = 1$  
    for  which $\conv(O(1))$ is a $1$-simplex.  
\end{proof}

Note that the orbit $O(n) \cdot {\rm Id}_n$ coincides  
with the orbit of the identity matrix ${\rm Id}_n$ under the action  
of the product group $O(n) \times O(n)$ by both right and left multiplication.  
Hence the tautological orbitope ${\rm conv}(O(n))$ is also an   
$O(n) \times O(n)$-orbitope for that action.  
We shall now digress and study these orbitopes in general.  
After we have seen (in Theorem  \ref{thm:OnOnspec})  
 that these are spectrahedra,  
we shall resume our discussion of ${\rm conv}(O(n))$.  
  
\subsection{Fan orbitopes} The group $G  
= O(n) \times O(n)$ acts on $\R^{n \times n}$ by simultaneous left and right  
translation. The action is given,   
for $(g,h) \in O(n) \times O(n)$ and $A \in \R^{n \times n}$, by  
\begin{equation}  
\label{eq:fanact}  
    (g,h) \cdot A \,\,:=\,\, g A h^T.  
\end{equation}  
Ky Fan  proved in \cite{Fan51} that the Schur-Horn theorem  for   
symmetric matrices under conjugation by $O(n)$    
generalizes to arbitrary square matrices   
under this $O(n)\times O(n)$ action.   
Now, singular values play the role of the eigenvalues.  
The following is a convex geometric reformulation:  
  
\begin{lemma}[Ky Fan \cite{Fan51}] \label{lem:Fan}  
    For a square matrix $A \in \R^{n \times n}$ let   
$\calO_A$ denote its orbitope under the action (\ref{eq:fanact})  
 of the group $O(n)  
    \times O(n)$. Then the image $\diag(\calO_A)$ of the projection to the  
    diagonal is the $B_n$-permutahedron with respect to the singular values  
    $\s(A)$.  
\end{lemma}  
  
We shall refer to  $\,\calO_A = \conv \{ (g,h)\cdot A  
    : g,h \in O(n) \}\,$ as the {\em Fan orbitope} of the matrix $A$.  
From Lemma \ref{lem:Fan},  
one easily deduces the analogous results to Theorems \ref{thm:facelatticeSH}  
and \ref{thm:facelatticeSkewSH}.

\begin{remark}  
    The facial structure of the Fan orbitope $\calO_A$  
is determined by the facial structure of the  
 $B_n$-permutahedron  $\BPi(\s(A))$ specified by  
 the singular values of the matrix~$A$.  
\end{remark}  
  
The description of the $B_n$-permutahedron  
in terms of weak majorization was stated in~\eqref{eq:Bpermineq}.  
Rephrasing these same linear inequalities for the   
singular values, and using  Lemma~\ref{lem:Fan},  
now leads to a spectrahedral   
description of the Fan orbitopes.  
For that we make use of the alternative characterization of   
singular values as the square roots of the (non-negative) eigenvalues of $AA^T$.  
Using Schur complements, it can be seen that the   
$2n \times 2n$-matrix  
\[  
    S(A) \,\,= \,\,\left( \begin{array}{ll}  
       0 & A \\  
        A^T & 0\\  
    \end{array}\right)  
\]  
has the eigenvalues $\pm \s_1(A), \dots, \pm \s_n(A)$.    
We form its additive compound matrices as before.  
  
\begin{theorem} \label{thm:OnOnspec}  
    Let $A$ be a real $n {\times} n$-matrix   
 with singular values $\s_1 \ge \s_2 \ge \cdots \ge  
    \s_n$.  
    Then its Fan orbitope $\calO_A$ equals the spectrahedron  
$$ \calO_A \,\, =\,\,\bigl\{    X \in \R^{n \times n}\,:\,  
 \Sfunc_k( S(X) )\,  \preceq \,   
(\s_1 + \cdots + \s_k) {\rm Id}_{\tbinom{2n}{k}}  
\quad k=1,\dotsc,n\,\bigr\}. $$  
\end{theorem}  
  
Fix an integer $p \in \{1,2,\ldots,n\}$. The  
 \emph{Ky Fan $p$-norm} is defined by  
\[  
    X \,\mapsto\, \sum_{i=1}^p \s_i(X).  
\]  
This function is indeed a norm on  $\R^{n \times n}$, and its  
unit ball is the Fan orbitope $\calO_A$  where  
$A$ is the diagonal matrix  
with $p$ diagonal entries $1/p$ and $n-p$ diagonal entries $0$.  
Theorem \ref{thm:OnOnspec} shows that the unit ball in the   
Ky Fan $p$-norm is a spectrahedron.  
Two norms of special interest in applied linear algebra  
are  the \emph{operator norm} $\| A \| :=  
\s_1(A)$ and  the \emph{nuclear norm}   
$\|A\|_{*} := \s_1(A) + \cdots + \s_n(A)$.   
Indeed, these two norms played the key role in  
work of Fazel, Recht and Parrilo  \cite{Fazel07}  
on compressed sensing in the matrix setting.  
Both the operator norm and the nuclear norm are closely tied  
to their vector counterparts.   
  
\begin{remark}  
    The unit balls in the operator and nuclear norm are both Fan orbitopes.  
    The projection to the diagonal yields the unit balls for the $\ell_\infty$  
    and the $\ell_1$-norm on $\R^n$. These two unit balls are polytopes  
in $\R^n$,  
namely, the $n$-cube and the $n$-crosspolytope respectively.  
\end{remark}  
  
Fazel {\it et al.} showed in \cite[Prop.~2.1]{Fazel07}  that the nuclear norm  
ball has the structure of a spectrahedron, and semidefinite programming duality  
allows for linear optimization  over the operator norm ball.  
The spectrahedral descriptions of both unit balls  
in $\R^{n \times n}$ are special instances of Theorem \ref{thm:OnOnspec},  
to be stated  explicitly once more.  
The operator norm ball consists of all matrices $X$   
whose largest singular value is at most $1$. This is equivalent to  
\begin{equation}  
\label{eq:Onspec}  
    \left(  
    \begin{array}{ll}  
        {\rm Id}_n & X \\  
        X^T & {\rm Id}_n \\  
    \end{array}  
    \right )  
    \succeq 0.  
\end{equation}  
The nuclear norm ball consists of all matrices $X$ for which the sum of the  
singular values is at most $1$. This is equivalent to saying that  
 the sum of the largest $n$ eigenvalues  
of the symmetric $\tbinom{2n}{n}\times\tbinom{2n}{n}$ matrix  
$ \Sfunc_n( S(X) )$  is at most $1$. Hence the nuclear norm ball  
in $\R^{n \times n}$ is the spectrahedron defined by the  
linear matrix inequality  
\begin{equation}  
\label{eq:Onspec2}  
   {\rm Id}_{\tbinom{2n}{n}} - \Sfunc_n( S(X) )   
   \, \succeq \,0.  
\end{equation}  
  
We are now finally prepared to return to the main aim of this section,  
which is the study of tautological orbitopes.  
The proof of Theorem \ref{thm:tautoOn} implies that   
the convex hull of $O(n)$ equals the  
set of $n {\times}n$-matrices whose largest singular value is at most $1$.  
In other words:  
  
\begin{corollary}  
The operator norm ball in $\R^{n \times n}$ is equal to the  
tautological orbitope of the orthogonal group $O(n)$.  
The coorbitope  ${\rm conv}(O(n))^\circ$ is the  
nuclear norm ball. Both of these convex bodies are  
spectrahedra. They are characterized in (\ref{eq:Onspec})  
and (\ref{eq:Onspec2}) respectively.  
\end{corollary}  
  
It would be worthwhile to explore  
implications of our geometric explorations of orbitopes  
for algorithmic applications in the sciences and engineering,  
such as those proposed in~\cite{Fazel07}.  
  
\subsection{The special orthogonal group}  
  
We next discuss the faces  
of the tautological orbitope of the group $SO(n)$.  
The relevant convex polytope is now  
the convex hull  $\hcube_n$ of all vertices $v \in \{-1, +1\}^n$   
with an even number of $(-1)$-entries.  This polytope is known as the  
\emph{demicube} or \emph{halfcube}. It is a permutahedron for the Coxeter  
group of type $D_n$, i.e.~signed permutations with an even number of sign changes.    
The facet hyperplanes of  $\hcube_n$ are derived by  
separating infeasible vertices of $[-1,+1]^n$ by hyperplanes through the   
$n$ neighboring vertices:  
\begin{equation}  
\label{eq:SOnfacets}  
\hcube_n \,\, = \,\,\bigl\{ x \in [-1,+1]^n \,:\,  
    \sum_{i \not\in J} x_i - \sum_{i \in J} x_i \le n-2  
\,\,\,\, \hbox{for all $J \subseteq [n]$ of odd cardinality} \bigr\}.  
\end{equation}  
While  $\hcube_2$ is only a line segment,   
we have ${\rm dim}(\hcube_n) = n$ for $n \geq 3$.  
For example, the  
halfcube $\hcube_3$ is the tetrahedron  
with vertices $(1,1,1)$, $(-1,-1,1)$,  
$(-1,1,-1)$ and $(1,-1,-1)$. Note that its facet inequalities   
appear as the diagonal entries in the symmetric  
$4 {\times} 4$-matrix (\ref{eq:magic}):  
$$  
\hcube_3 \,=\,  
\bigl\{ \,x\in \R^3\,:\  
{\rm min} (  
1+x_1 + x_2 + x_3 ,\,  
1+x_1 - x_2 - x_3 ,\,  
1 - x_1 + x_2 - x_3 ,\,  
1 - x_1 - x_2 + x_3 ) \geq 0 \,\bigr\}.  
$$  
This observation is explained by results of Horn (cf.~\cite{Leite99}) on the  
diagonals of special orthogonal matrices. These imply the  
following lemma about the tautological orbitope of $SO(n)$:  
  
\begin{lemma}   
The projection of  
 $\,{\rm conv}(SO(n))$ onto the diagonal equals  
the halfcube $\hcube_n$.  
\end{lemma}  
  
\begin{proof}  
Just like in Lemma \ref{lem:Ontodiagonal}, it is clear that  
 $\hcube_n$ is a subset of $ \diag({\rm conv}(SO(n))$.  
The converse is derived from the linear algebra fact that the  
trace of any matrix in $O(n)\backslash SO(n)$ is at most $n-2$.  
For  $J \subseteq [n]$ let $R_J$ be the diagonal matrix with   
$(R_J)_{ii} = -1$ if $i \in J$ and  
$(R_J)_{ii} = 1$ if $i \not\in J$.   
Let $g \in SO(n)$. Then  
${\rm trace}(g \cdot R_J) \leq n-2$ for all $J$  
of odd cardinality. This means that  
$\diag(g)$ satisfies the linear inequalities in  
(\ref{eq:SOnfacets}) and hence lies in $\hcube_n$.  
\end{proof}  
  
There is a variant of singular value decomposition with respect to the  
restricted class of orientation preserving transformations. For every   
matrix $A \in \R^{n \times n}$ there exist rotations $U,V \in SO(n)$ such that  
$UAV$ is diagonal. The diagonal entries are called the \emph{special singular  
values} and denoted by $\ts(A) = (\ts_1(A) \ge \cdots \ge \ts_n(A))$. The main  
difference to the usual singular values is that $\ts_n(A)$ may be negative;  
only the first $n-1$ entries of $\ts(A)$ are non-negative.  
We need to make this distinction in   
order to understand the faces of ${\rm conv}(SO(n))$.

\begin{theorem} \label{thm:SOntauto}  
    The tautological orbitope of $SO(n)$ has precisely two orbits of  
    facets. These are the tautological orbitopes for $SO(n-1)$  
and the free spectrahedra of dimension $\tbinom{n}{2}-1$.  
\end{theorem}  
   
\begin{proof}  
Up to $D_n$-symmetry, the halfcube   
$\hcube_n$ has only two  
   distinct facets, namely an $(n-1)$-dimensional halfcube and an  
   $(n-1)$-dimensional simplex. A typical halfcube facet $(\hcube_n)^\ell$   
arises by maximizing the linear function $\ell = x_1$ over  
$\hcube_n$, and a typical simplex  
facet $(\hcube_n)^{\ell'}$   
 arises by maximizing the linear function  
$\,\ell'= x_1 + x_2 + \cdots + x_{n-1} - x_n$.  
   Pulling back $\ell$ along the diagonal projection  $\diag$, we  
see that the facet of ${\rm conv}(SO(n))$ corresponding to the  
halfcube facet is  
the convex hull of all rotations $g \in SO(n)$ that fix the  
   first standard basis vector $e_1$.  
   Pulling back $\ell'$ along  $\diag$, we see that the facet of  
 ${\rm conv}(SO(n))$ corresponding to the simplex facet is the convex hull  of  
$\,\bigl\{g \in SO(n) \,:\,\Tr(g \cdot R_{\{n\}}) = n-2 \bigr\}$.  
This facet is isomorphic to the convex  
   hull of all $g^\prime \in O(n) \backslash SO(n)$ such that $\Tr(g^\prime) =  
   n-2$.  Since $g^\prime$ is orientation reversing, one eigenvalue of  
$g^\prime$ is $-1$, and $\Tr(g^\prime) = n-2$ forces all other eigenvalues  
to be equal to $1$. Hence the facet in question is the  
symmetric Schur-Horn orbitope for the diagonal matrix  
$(1,\ldots,1,-1)$. Example \ref{ex:freespec} implies that  
 this is a free spectrahedron.  
\end{proof}

In Subsection 4.1 we exhibited  a spectrahedral  
representation for the tautological orbitope   
${\rm conv}(SO(3))$, and in that case, the  
two facet types of Theorem \ref{thm:SOntauto}  
collapse into one type. At present, we do not  
know how to generalize the representation  
(\ref{eq:magic}) to $SO(n)$ for $n \geq 4$.

\section{Carath\'eodory Orbitopes}   
   
Orbitopes for $SO(2)$ were first  
studied by Carath\'e\-o\-dory \cite{Car}.  
The coorbitope cone (\ref{eq:coorbitopecone})  
dual to such a {\em Carath\'eodory orbitope}   
consists of non-negative trigonometric polynomials.   
This leads to the Toeplitz spectrahedral representation of the universal   
Carath\'eodory orbitope in  
Theorem \ref{Th:toeplitz}, which implies that the  
convex hulls of all trigonometric curves are projections of spectrahedra.  
The universal Carath\'eodory orbitope is also affinely isomorphic to the convex hull of  
the compact even moment curve,  whose coorbitope cone   
consists of non-negative univariate polynomials.    
This leads to the   
 representation by Hankel matrices in Theorem~\ref{Th:Hankel}.

%%%%%%%%%%%%%%%%%%%%%%%%%%%%%%%%%%%%%%%%%%%%%%%%%%%%%%%%%%%%%%%%%%%%%%%%%%%%%   
\subsection{Toeplitz representation} \label{S:Caratheodory}  
   
The irreducible representations $\rho_a$ of $SO(2)$ are indexed by   
non-negative integers $a$. Here,  
$\rho_0$ is the trivial representation.   
When $a \in \N$ is positive, the representation $\rho_a$ of $SO(2)$  
acts on $\R^2$, and it  
sends a rotation matrix to its $a$th power:  
 $$   
  \rho_a \,\,:\,\,   
   \begin{pmatrix}    
    \cos(\theta) &  - \sin(\theta) \, \\   
     \sin(\theta) & \phantom{-}\cos(\theta) \,   
   \end{pmatrix}  
  \,\,\, \mapsto \,\,\,   
   \begin{pmatrix}    
    \phantom{-} \cos(\theta) & - \sin(\theta) \, \\   
      \sin(\theta) & \phantom{-}\cos(\theta) \,   
    \end{pmatrix}^{\! a}  
    \, = \,   
    \begin{pmatrix}    
      \cos( a \theta) & -\sin( a \theta) \, \\   
      \sin(a \theta) & \phantom{-}\cos( a \theta) \,   
    \end{pmatrix} .  
 $$   
For a vector $A = (a_1,a_2,\ldots,a_d) \in \N^d$ we consider   
the direct sum of these representations  
$$  
   \rho_A \,\,\, := \,\,\,    
  \rho_{a_1} \,\oplus \, \rho_{a_2} \,\oplus \,\cdots \, \oplus \,\rho_{a_d}.   
$$   
The {\em Carath\'eodory orbitope}  $\Cara_A$ is the convex hull   
of the orbit $SO(2)\cdot (1,0)^d$ under the action $\rho_A$  
 on the vector space $(\R^2)^d$.  
This orbit is the  {\em trigonometric moment curve}   
\begin{equation}   
 \label{CaraCurve}   
  \Bigl\{ \bigl( \,   
   \cos(a_1 \theta) , \,   
   \sin(a_1 \theta) , \,   
%   \cos(a_2 \theta) , \,   
%   \sin(a_2 \theta) , \, 
    \ldots \,,   
   \cos(a_d \theta) , \,   
   \sin(a_d \theta) \,\bigr) \in \R^{2d} \,: \,   
   \theta \in [0,2\pi) \,   
  \Bigr\}.   
\end{equation}   
This curve is also identified with the matrix group $\,\rho_A(SO(2))$ lying in  
the space $(\R^{2 \times 2})^d$ of block-diagonal $2d \times 2d$-matrices with $d$ blocks of  
size $2 \times 2$.    
Thus $\Cara_A$ is isomorphic to the convex hull of   
$\,\rho_A(SO(2))$, and can therefore also be thought of as   
a tautological orbitope.  
  
We distinguish between isomorphisms of orbitopes that preserve the $SO(2)$-action, and the  
weaker notion of affine isomorphisms that preserve their structure as convex bodies.  
  
%%%%%%%%%%%%%%%%%%%%%%%%%%%%%%%%%%%%%%%%%%%%%%%%%%%%%%%%%%%%  
\begin{lemma}\label{L:all_Cara}  
 Any orbitope of the circle group   
$SO(2)$ is isomorphic to a Carath\'eodory orbitope  
 $\Cara_A$, where  $A\in\N^d$ has distinct coordinates, and it is   
 affinely isomorphic to a  Carath\'eodory orbitope   
where the coordinates of $A$ are  
 relatively prime integers.  
\end{lemma}  
%%%%%%%%%%%%%%%%%%%%%%%%%%%%%%%%%%%%%%%%%%%%%%%%%%%%%%%%%%%%  
\begin{proof}  
 Let $\calO$ be an orbitope for $SO(2)$.  
 We may assume that its ambient $SO(2)$-module $V$  
 has no trivial components and is the linear span of $\calO$.   
 Then $V$ has the form $\rho_A$ for $A\in \N^d$ with distinct non-zero components, that  
 is, $V$ is multiplicity free.   
 This is because $\End_{SO(2)}(\rho_a)=\C$, where we identify  
 $\R^2$ with $\C$  and $SO(2)$ with the unit circle in $\C$.  
 The orbitope $\calO$ is generated by a vector $v=(v_1,\dotsc,v_d)\in\C^d$   
with non-zero coordinates. By complex rescaling,   
we may assume that $v=(1,\dotsc,1)$, showing that   
 $\calO$ is isomorphic to $\Cara_A$.  
 Lastly, if the coordinates   
of $A=(a_1,\dotsc,a_d)$ have greatest common divisor $a$, then  
 $\rho_A$ is the composition $\rho_{A'}\circ \rho_a$, where $A'=(a_1/a,\dotsc,a_d/a)$ and  $\calO$ is affinely isomorphic to an orbitope for the   
module $\rho_{A'}$.  
\end{proof}  
  
We henceforth assume that $0<a_1<\dotsb<a_d$ where the $a_i$  
 are relatively prime.   
When $A=(1,2,\dotsc,d)$,  Carath\'eodory \cite{Car} studied the facial structure of   
$\Cara_A$. 
An even-dimensional cyclic polytope is the convex hull of finitely many points on  
Carath\'eodory's curve (see~\cite{Ziegler}).  
Smilansky \cite{Smi} studied the four-dimensional Carath\'eodory orbitopes   
($d=2$), and this was recently extended by   
Barvinok and Novik \cite{BN} who studied $\Cara_{(1,3,5,\ldots,2k{-}1)}$.   
The corresponding curve (\ref{CaraCurve}) is the   
{\em symmetric moment curve} which   
gives rise to a remarkable family of centrally symmetric polytopes    
with extremal face numbers.   
Many questions remain about the facial structure of the   
Barvinok-Novik orbitopes  $\Cara_{(1,3,5,\ldots,2k{-}1)}$.  See  
\cite[\S 7.4]{BN} for details.  
   
We now focus on the {\em universal Carath\'eodory orbitope}   
$\Cara_d := \Cara_{(1,2,\dots,d)} $ in $\R^{2d}$.  
This convex body has the following  
 spectrahedral representation in terms of Hermitian Toeplitz matrices.  
  
%%%%%%%%%%%%%%%%%%%%%%%%%%%%%%%%%%%%%%%%%%%%%%%%%%%%%%%  
\begin{theorem}\label{Th:toeplitz}  
    The universal Carath\'eodory orbitope $\Cara_d $ is isomorphic to   
         the spectrahedron  consisting of   
          positive semidefinite Hermitian Toeplitz matrices with ones  
    along the  diagonal:  
\begin{equation}  
\label{toeplitz}  
     \left(   
        \begin{array}{ccccc}   
                1   & x_1   & \cdots  & x_{d-1} & x_d    \\   
            y_1     & 1   & \ddots  &   x_{d-2}      & x_{d-1} \\   
            \vdots  & \ddots  & \ddots  & \ddots    & \vdots \\   
            y_{d-1} &  y_{d-2}       & \ddots  & 1    & x_{1}\\   
            y_d     & y_{d-1} & \cdots  &   y_{1}   & 1  \\   
        \end{array}\right) \,\succeq \, 0  
         \qquad \ \mbox{where}\qquad  
\begin{bmatrix}  
          x_j = c_j +  s_j \cdot i \\  
          y_j  = c_j -  s_j \cdot i \\  
           i \, = \, \sqrt{-1}  
 \end{bmatrix}  
\end{equation}  
\end{theorem}   
  
\smallskip  
  
We note that the complex spectrahedron (\ref{toeplitz})  
can be translated into a spectrahedron over $\R$  as follows.  
Consider a Hermitian matrix $\,H = F +   G \cdot i\,$ where $F$  
is real symmetric and $G$ is real skew-symmetric.  
The Hermitian matrix $H$ is positive  
definite if and only if  
\[   
    \left(\begin{array}{rr}   
        F & -G \\   
        G & F \\   
    \end{array}   
    \right)  \succeq 0.   
\]

A {\em trigonometric curve} is any curve in $\R^n$  
that is parametrized by polynomials in   
the trigonometric functions sine and cosine, or equivalently,  
any curve that is the image under a linear map  
 of the universal trigonometric moment   
curve~\eqref{CaraCurve} where $A=(1,2,\dotsc,d)$.  
  
%%%%%%%%%%%%%%%%%%%%%%%%%%%%%%%%%%%%%%%%%%%%%%%%%%%%%%%  
\begin{corollary}[cf.~Henrion \cite{henrion}]  
 The convex hull of any trigonometric curve is a projected spectrahedron.  
 In particular, all Carath\'eodory orbitopes are projected spectrahedra.  
\end{corollary}  
 
\begin{figure}[htb]  
\[  
  \begin{picture}(400,135)(0,-20)  
   \put(  0,0){\includegraphics[height=120pt]{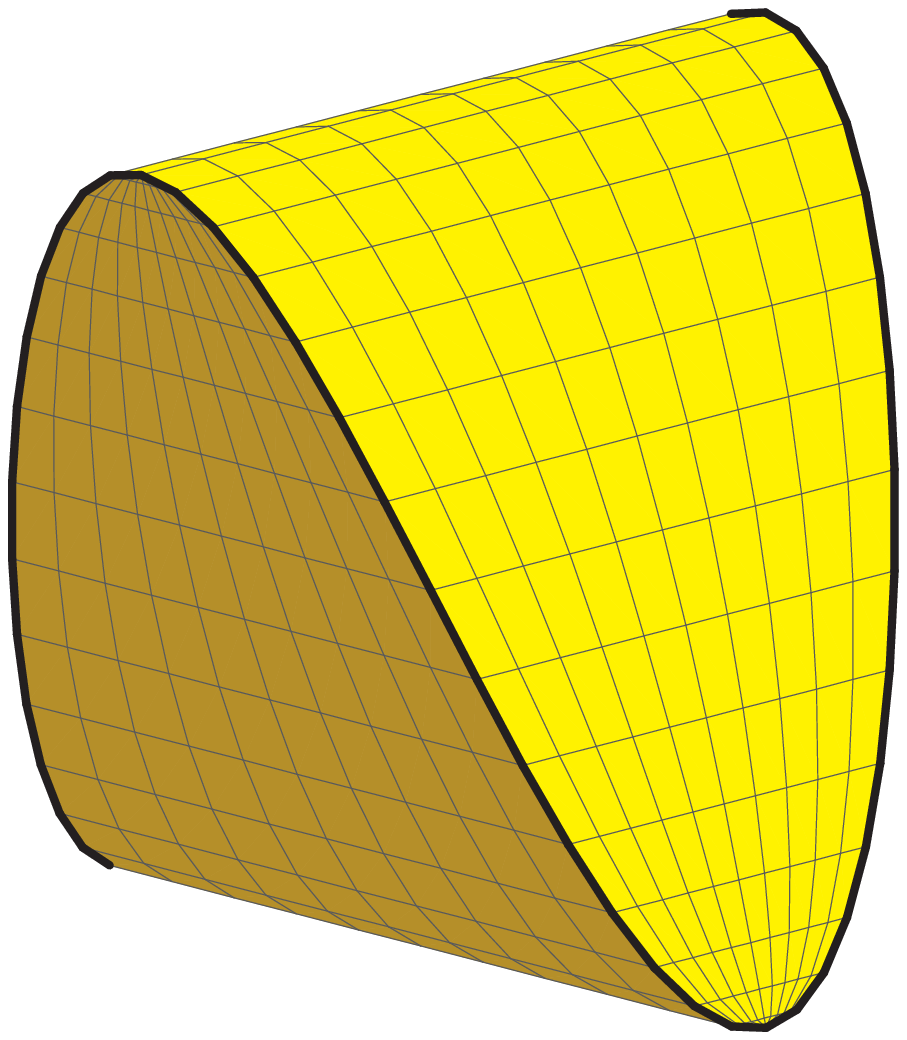}}  
   \put(135,0){\includegraphics[height=120pt]{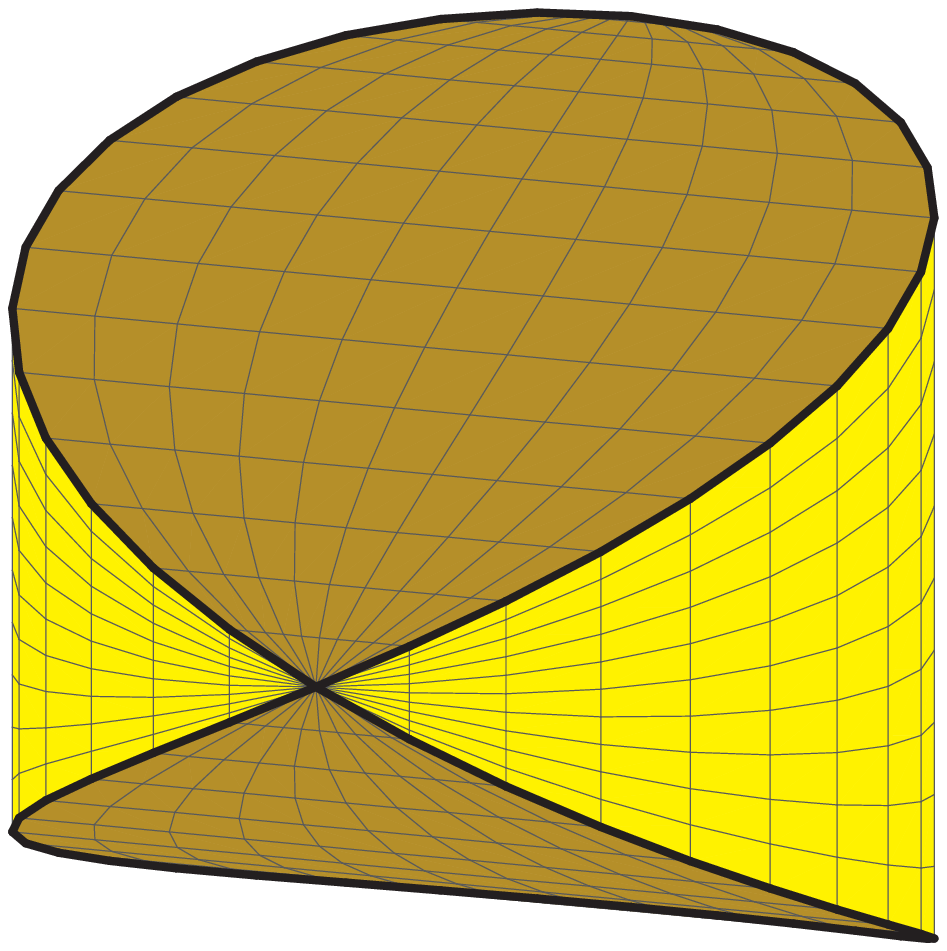}}  
   \put(280,0){\includegraphics[height=120pt]{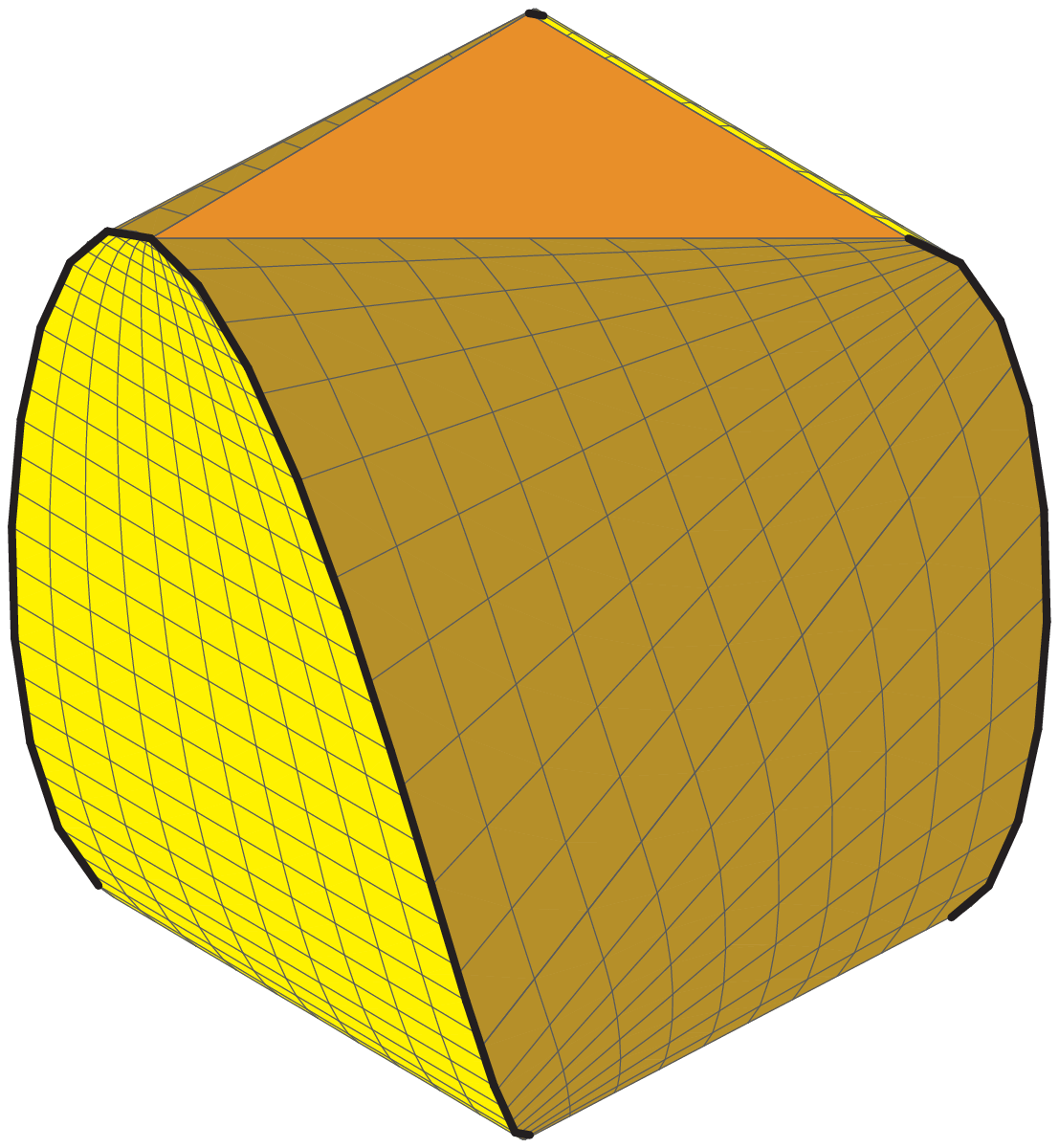}}  
   \put(  0,-17){$(\cos(\theta),\sin(\theta),\cos(2\theta))$}  
   \put(135,-17){$(\cos(2\theta),\sin(2\theta),\cos(\theta))$}  
   \put(280,-17){$(\cos(\theta),\sin(2\theta),\cos(3\theta))$}  
  \end{picture}  
\]  
\vskip -0.3cm  
\caption{Convex hulls of three trigonometric curves.}  
\label{F:trig_curves}    
\end{figure}

Figure~\ref{F:trig_curves} shows the convex hulls of three   
trigonometric curves in $\R^3$. The left and middle convex bodies are each the intersection
of two convex quadratic cylinders
($2x^2 = 1 + z$ and $2y^2 = 1 - z$ for the former; $x^2 + y^2 = 1$ and 
$2z^2 = 1 + x$ for the latter) and hence are spectrahedra. 
The rightmost convex body is  
visibly not a spectrahedron. Its exposed points are  
$(\cos(\theta),\sin(2\theta),\cos(3\theta))$ for    
$\theta\in(-\frac{\pi}{3},\frac{\pi}{3})\cup (\frac{2\pi}{3},\frac{4\pi}{3})$,  
%  
% Frank: After some reflection, I think that the endpoint values of   
%  \theta are also not exposed  
%  
and it has two algebraic  families of one-dimensional facets.  
In addition, there are two two-dimensional facets,  
namely the equilateral triangles for  
$\,\theta = 0, \frac{2 \pi}{3}, \frac{4 \pi}{3}$ and  
$\,\theta = \pi, \frac{\pi}{3}, - \frac{\pi}{3}$.  
The six edges of these two triangles are one-dimensional non-exposed faces,  
as there is no linear function which achieves its minimum on this body   
along these edges.  Also, exactly one vertex of each triangle at $\theta=0$ and
$\theta=\pi$ is exposed.
We conclude that this convex body  
is a projected spectrahedron but it is not a spectrahedron.

%%%%%%%%%%%%%%%%%%%%%%%%%%%%%%%%%%%%%%%%%%%%%%%%%%%%%%%%%%%%%%%  
\begin{proof}[Proof of Theorem~\ref{Th:toeplitz}]  
The coorbitope cone dual to $\Cara_d$ consists of all  
affine-linear functions that are non-negative on $\Cara_d$.    
These correspond to  
non-negative trigonometric polynomials:  
\[   
    \widehat{\Cara}^\circ_d  
 \,= \,\bigl\{ (\delta,a_1, b_1,\dots, a_d, b_d) \in \R^{2d+1} :   
    \delta + \sum_{k=1}^d a_k \cos(k\theta) + b_k \sin(k\theta) \ge 0 \text{   
    for all }\theta \bigr\} .  
\]   
We identify each point $ (\delta,a_1, b_1,\ldots, a_d, b_d) $ in $ \R^{2d+1}$  
with the Laurent polynomial  
 \begin{equation}\label{Eq:trig_poly}   
   R(z) \,\,= \, \sum_{k=-d}^d u_k z^k \in \C[z,z^{-1}]   
 \end{equation}  
with $u_0 = \delta$, $u_k = \tfrac{1}{2}(a_k -   b_k i)$,  
$i=\sqrt{-1}$, and $u_{-k} =   
\overline{u_k}$.   
Then $\overline{R(z)}=R(\overline{z}^{-1})$ and $R\in\widehat{\Cara}^\circ_d$ if and only if  
$R$ is non-negative on the unit circle $\sphere^1$ of $\C$.  
Roots of $R$ occur in pairs $\alpha,\overline{\alpha}^{-1}$ and those on   
$\sphere^1$  
have even multiplicity.  
Choosing one root from each pair gives the factorization  
 \begin{equation}  
 \label{circlesos1}  
        R(z) \,\, = \,\, \overline{H}(z^{-1}) \cdot H(z),   
 \end{equation}  
where $H\in\C[z]$ has degree $d$ and the coefficient vectors of $H$ and $\overline{H}$ are  
complex conjugate.  
This factorization is the classical Fej\'er-Riesz Theorem.  
   
Utilizing the monomial map $\gamma_d : \C \rightarrow \C^{d+1}$ with   
$\gamma_d(z) = (1,z,z^2,\dots,z^d)^T$, this is equivalent to the following:  
    A trigonometric polynomial $R(z)$ is non-negative on the unit circle if and   
    only if there is a non-zero vector $h\in \C^{d{+}1}$ such that   
%  
%  The existence of a PD hermitian matrix was not clear to Frank that this was equivalent to  
%  
% "a positive semidefinite Hermitian matrix $H \in \C^{(d+1)\times (d+1)}$ such that"  
%  
%  Why if such a representation exists, is there one with rank$H=1$?  
%  
\begin{equation}  
\label{circlesos2}  
        R(z) \,\, = \,\, \gamma_d(z^{-1})^T\cdot \overline{h}h^T \cdot \gamma_d(z).   
\end{equation}  
  
A point $(c_1,s_1,\dots,c_d,s_d) \in \R^{2d}$ belongs to  
the Carath\'eodory orbitope $\,\Cara_d\,$ if and only if    
\begin{equation}  
\label{circlepairing}  
 \,\,  
    \delta + \sum_{k=1}^d a_k c_k + b_k s_k \, \ge \, 0 \quad \,\,  
\hbox{for all $\, (\delta,a_1,b_1,\dots,a_d,b_d) \in \widehat{\Cara}^\circ_d$.}  
\end{equation}  
The sum on the left equals the Hermitian  
inner product in $\C^{2d+1}$ of the coefficient vector  
$u$ of the polynomial $R(z)$ and the vector  
$\zeta=(x,1,y)$ with $x_k,y_k$ as in (\ref{toeplitz}).  
The formula (\ref{circlesos2}) expresses $u$  
as the image of the Hermitian matrix $\overline{h}h^T$ under some linear projection $\pi$.   
If $\pi^*$ denotes the linear map dual to $\pi$ then   
$X = \pi^*(\zeta) $ is precisely the Hermitian Toeplitz matrix in (\ref{toeplitz}).  
We conclude that the sum in (\ref{circlepairing}) equals  
\[  
     \langle \zeta, \pi(\overline{h}h^T) \rangle \,=\,   
    \langle \pi^*(\zeta), \overline{h}h^T \rangle \,=\,   
    \Tr (X \cdot \overline{h}h^T) \,=\,  h^T\cdot X \cdot \overline{h}.  
\]   
Thus the point $(c_1,s_1,\ldots,c_d,s_d)$ represented by  
a Hermitian Toeplitz matrix $X$ lies in $\Cara_d$ if and only  
if $\,h^T\cdot X \cdot \overline{h} \geq 0\,$ for all $\,h \in \C^{d+1}\,$    
if and only if $X$  
is positive semidefinite.    
\end{proof}   
  
The proof of Theorem~\ref{Th:toeplitz} elucidates  
the known results about the  
facial structure of $\Cara_d$.  
  
\begin{corollary}\label{neighborly}  
    The universal Carath\'eodory  orbitope  
$\mathcal{C}_d$ is a neighborly simplicial convex body.   
    Its faces are in inclusion-preserving correspondence with sets   
    of at most $d$ points on the circle.   
\end{corollary}   
  
\begin{proof}   
 A Laurent polynomial $R$ as in~\eqref{Eq:trig_poly} lies in the boundary of  
 the coorbitope cone $\widehat{\Cara}^\circ_d$ if and only if it is  
 non-negative on the unit circle $\sphere^1$ 
 but not strictly positive. 
It supports the face of $\Cara_d$ spanned by the points of the trigonometric  
 moment curve corresponding to its zeros in $\sphere^1$.  Each zero has  
 multiplicity at least $2$, so there are at most $d$ such points, and   
conversely  any subset of $\leq d$ points supports a face.    
Since any fewer than $2d+2$ points on  
 the curve are affinely independent and since all faces  
 are exposed, these faces are simplices.   
\end{proof}   
  
Corollary \ref{neighborly} implies that  
 the Carath\'eodory number $\cara(\Cara_d)$ equals $d+1$, as no point in the  
 interior of $\Cara_d$ lies in the convex hull of $d$ points of the orbit, but  $\cara(\Cara_d)$  
 is at most one more than the maximal Carath\'eodory number of a facet, by Lemma~3.2  
 of~\cite{LSS08}.

Carath\'eodory orbitopes are generally not spectrahedra  
because they can possess non-exposed faces.  
Smilansky \cite{Smi} showed that if we write  
$\rho(\theta)\in\R^4$ for a point on the trigonometric moment curve with weights $1$ and $3$  
then the faces of $\Cara_{(1,3)}$   
are exactly the points $\rho(\theta)$ of the orbit, the line segments  
${\rm conv}\{\rho(\theta),\rho(\theta+\alpha)\}$, where $0<\alpha<\frac{2\pi}{3}$, and  
the triangles  
\[{\rm conv}\{  
    \rho(\theta), \   
    \rho(\theta+2\pi/3), \ 
    \rho(\theta+4\pi/3)\}.  
\]  
In particular, each line segment  
${\rm conv}\{\rho(\theta),\rho(\theta+\frac{2\pi}{3}\}$ is a non-exposed  
edge of $\Cara_{(1,3)}$. We conclude that the  
Barvinok-Novik orbitope $\Cara_{(1,3)}$  is not a spectrahedron.  
%See also \cite[\S 2.4]{Ranestad2}. 
  
 The Toeplitz representation (\ref{toeplitz}) of the universal   
Carath\'eodory orbitope  
 $\Cara_d$ reveals complete algebraic information. For example, the   
 algebraic boundary $\partial_a \Cara_d$ is   
the irreducible hypersurface of degree $d+1$ defined by the determinant of  
 that  $(d{+}1) \times (d{+} 1)$-matrix.   
 The  curve~\eqref{CaraCurve} itself is the  
 set of all positive definite Hermitian Toeplitz matrices  
 of rank one.   
The $2 \times 2$-minors of the matrix (\ref{toeplitz})  
generate the  prime ideal  $J_{(1,\ldots,d)}$ of this rational curve.  
   
The union of the $(j-1)$-dimensional faces  of $\Cara_d$  
is the set of positive definite  
 Hermitian Toeplitz matrices of rank $j$, as a point lies on a  $(j-1)$-dimensional face if and  
 only if it is the convex combination of $j$ points of the curve.  
The Zariski closure of this stratum is the set of all rank $j$ Hermitian Toeplitz matrices  
 which is defined by the vanishing of the  $(j{+}1)\times(j{+}1)$ minors of that matrix.  
This is also the $j$th secant variety of the Carath\'eodory curve.  
Lastly, this rank stratification is a Whitney stratification of the   
algebraic boundary.  
  
The derivation of the algebraic description of  
$\Cara_A$ for arbitrary $A = (a_1,a_2,\ldots,a_d)$ requires  
the process of elimination. For instance, the ideal   
$J_A$ of the trigonometric moment curve (\ref{CaraCurve})  
can be computed from the ideal of $2 {\times} 2$-minors  
for $J_{(1,2,\ldots,a_d)}$ by eliminating all  
unknowns $x_j, y_j$ with $j \not\in \{a_1,\ldots,a_d\}$.  
The equation of the algebraic boundary $\partial_a \Cara_A$ is  
obtained by the same elimination applied to a certain  
ideal of larger minors of the Toeplitz  
matrix (\ref{toeplitz}).  
  
We refer to recent work of Vinzant \cite{Vinz} 
for a detailed study of the edges of  
Barvinok-Novik orbitopes. 
An analysis of the algebraic boundary the orbitope 
$ \Cara_{(1,3)} $ can be found 
in \cite[\S 2.4]{Ranestad2}.

%%%%%%%%%%%%%%%%%%%%%%%%%%%%%%%%%%%%%%%%%%%%%%%%%%%%%%%%%%%%%%%%%%%%%%%%%%%%%%%%%  
 \subsection{Hankel representation}\label{S:Hankel}

The cone over the degree $d$ moment curve is the image of $\R^2$ in its $d$th symmetric  
power $\mbox{Sym}_d \R^2\simeq \R^{d+1}$ under the map  
\[  
   \nu_d\ \colon\ (x,y)\ \longmapsto\   
    \bigl(\,x^d,\, x^{d-1}y,\, x^{d-2}y^2,\,\dotsc,\, y^d\,\bigr)\,.  
\]  
This map is naturally $SO(2)$-equivariant.  
We define the {\em compact moment curve} to be the image $\nu_d(\sphere^1)$   
of the unit circle under the map $\nu_d$. This restricted map equals  
\[  
   \sphere^1\ni \theta\ \longmapsto\   
   (\cos^d(\theta), \cos^{d-1}(\theta)\sin(\theta), \dotsc, \sin^d(\theta))\,.  
\]  
The convex hull of the curve $\nu_d(\sphere^1)$ is an orbitope.  
By Lemma~\ref{L:all_Cara}, it is  
isomorphic to some Carath\'eodory orbitope $\Cara_A$.  
The following lemma makes this identification explicit.  
  
\begin{lemma}\label{L:identification}  
 If $d \in \N$ is odd, then  $\conv(\nu_d(\sphere^1))$ is isomorphic to the Barvinok-Novik  
 orbitope $\Cara_{(1,3,\dotsc,d)}$.   
 If $d \in \N$ is even, then $\conv(\nu_d(\sphere^1))$ is isomorphic to   
$\Cara_{(0,2,4,\dotsc,d)}$,  
 which is affinely isomorphic to the universal Carath\'eodory orbitope $\Cara_{d/2}$.  
\end{lemma}  
%%%%%%%%%%%%%%%%%%%%%%%%%%%%%%%%%%%%%%%%%%%%%%%%%%%%%%%%%%%%%%%  
  
\begin{proof}  
  Complexifying the $SO(2)$-module $\rho_A$  where   
  $A=(a_1,\dotsc,a_d)$ gives the $\C^\times$-module with   
symmetric weights   
  $\pm a_1, \dotsc,\pm a_d$.  
  Thus the underlying real $SO(2)$-module of this $\C^\times$-module   
 is $\rho_{(|a_1|,\dotsc,|a_d|)}$.  
 The lemma follows because the complexified   
representation $\mbox{Sym}_d \C^2$ of   
  $\mbox{Sym}_d \R^2 $ has weights $d, d{-}2, d{-}4, \dotsc, -d$, and  
 this representation is  
  spanned by the pure powers of linear forms, so every weight appears in the linear span of the  
  orbit.   
\end{proof}  
  
Suppose now that $d=2n$ is even.  
We describe the moment curve and the Carath\'eodory orbitope in   
coordinates  
$(\l_0,\l_1,\dotsc,\l_{2n})$ for $\R^{2n+1}$.  
Fix the  $(n{+}1) {\times} (n{+}1)$-Hankel matrix   
\begin{equation}  
 \label{Eq:Hankel}  
   K(\l) \quad = \quad \begin{pmatrix}   
      \l_0 & \l_1 & \l_2 & \cdots & \l_n \\   
      \l_1 & \l_2 & \l_3 & \cdots  & \l_{n+1} \\    
      \l_2  & \l_3 & \l_4 & \cdots & \l_{n+2} \\   
       \vdots & \vdots  & \vdots &  & \vdots \\   
       \l_n & \l_{n+1} & \l_{n+2} & \cdots & \l_{2n}   
   \end{pmatrix}\ .  
\end{equation}  
   
\begin{theorem} \label{Th:Hankel}  
The even moment curve consists of all vectors   
$\l \in \R^{2n+1}$ such that the Hankel matrix   
$K(\l)$ has rank one, is positive semidefinite,   
and satisfies the linear equation   
\begin{equation}   
\label{linrel}   
 \sum_{j=0}^n \binom{n}{j} \l_{2j} \,\, = \,\, 1 .    
 \end{equation}   
Its convex hull consists of all $\l$ such that   
$K(\l)$ is positive semidefinite and satisfies~\eqref{linrel}.   
\end{theorem}   
   
Thus the Carath\'eodory orbitope $\Cara_n$ has a second  
{\em Hankel representation} as a  
spectrahedron.  
%While spectrahedral representations are not unique (they  
% depend in part on choices of  
% coordinates), it is striking that $\Cara_n$ has two very different,  
%yet both highly  structured representations.  
 The proof of this theorem follows from the well-known fact   
that every non-negative polynomial in one variable is a sum   
of squares of polynomials. It uses the duality in \cite[\S 3]{Rez}.   
  
\begin{proof}  
 Observe that the points of the even compact moment curve  
 satisfy~\eqref{linrel}, which comes from the polynomial identity   
 $(x^2+y^2)^n=1$.  As $\mbox{Sym}_{2n}\R^2$ has just one copy of the  
 trivial representation, this is the only affine equation that holds on  
 $\nu_{2n}(\sphere^1)$.  The dual to $\mbox{Sym}_{2n}\R^2$ is the space  
 $\R[x,y]_{2n}$ of real homogeneous polynomials of degree $2n$ in $x$ and $y$.  
 The coefficients of such a polynomial gives coordinates for $\R[x,y]_{2n}$.  
  
The coorbitope cone $\widehat{\nu_{2n}(\sphere^1)^\circ}$   
dual to the orbitope $\conv(\nu_{2n}(\sphere^1))$ is the cone  
 of homogeneous polynomials of degree $2n$ that are non-negative on $\R^2$.  
 Thus a point $(\l_0,\dotsc,\l_{2n})\in\mbox{Sym}_{2n}\R^2$ lies in the orbitope  
 $\conv(\nu_{2n}(\sphere^1))$ if and only if it satisfies~\eqref{linrel} and  
\begin{equation}\label{Eq:nonnegsum}  
   \sum_{i=0}^{2n} f_i \l_i\ \geq\ 0\, 
\end{equation}  
 for every non-negative polynomial $f(x,y)=\sum_i f_ix^iy^{2n-i}$.  
  
 Since non-negative homogeneous polynomials in $x$ and $y$   
are sums of squares  
(cf.~\cite{Rez}), we only need these inequalities to  
 hold when $f(x,y)=g(x,y)^2$ is a square.  
 Writing $g=(g_0,\dotsc,g_{n})$ for the coefficient vector of the polynomial $g(x,y)$,   
 the sum~\eqref{Eq:nonnegsum} becomes  
\[  
   \sum_{i=0}^{2n} \l_i \sum_{j+k=i} g_j g_k \,\,\, = \,\,\,  
    g^T \cdot K(\l)  \cdot g\,,  
\]  
 where $K(\l)$ is the Hankel matrix~\eqref{Eq:Hankel}.  
 This proves the theorem.  
\end{proof}  
  
%%%%%%%%%%%%%%%%%%%%%%%%%%%%%%%%%%%%%%%%%%%%%%%%%%%%%%%%%%%%%%%%%%%%%  
  
\section{Veronese Orbitopes}  
   
The Hankel representation of the universal Carath\'eodory orbitope   
arose by considering the image of the circle $\sphere^1\subset\R^2$   
in $\mbox{Sym}_{2n}\R^2$ and its relation to non-negative binary forms.  
Generalizing from $\R^2$ to $\R^d$ gives the Veronese orbitopes  
 whose coorbitope cones (\ref{eq:coorbitopecone})  
 consist of non-negative $d$-ary forms.  
When $d\geq 3$, non-negative forms are not necessarily sums of squares, except for  
quadratic forms and the exceptional case of ternary quartics.  
%We study this exceptional Veronese orbitope in detail. 
% from the perspective of  convex algebraic geometry.  
  
The set of  decomposable symmetric tensors is the image of the  
 Veronese map   
\[  
  \nu_{m}\ \colon\ \R^d\ \longrightarrow\   
  \mbox{Sym}_m\R^d \,\simeq \,\R^{\binom{d+m-1}{m-1}}\,.  
\]  
The $SO(d)$-orbits through any two non-zero decomposable tensors are scalar  
multiples of each other and are thus isomorphic.  
We define the {\em Veronese orbitope} $\mathcal{V}_{d,m}$ to   
be the convex hull of the orbit through the specific decomposable  
tensor $\nu_{m}(1,0,\dotsc,0)$. That orbit is   
also the image $\nu_{m}(\sphere^{d-1})$  
of the unit $(d-1)$-sphere under the   
$m$-th Veronese embedding of $\R^d$.   
  
Suppose that $m=2n$ is even.  
Then the orbit $\nu_{m}(\sphere^{d-1})$  
can be identified with $\R\PP^{d-1}$ since  
 $\nu_{2n}$ is two-to-one with $\nu_{2n}(v)=\nu_{2n}(-v)$.  
The dual vector space to $\mbox{Sym}_{2n}\R^d$ is   
the space of homogeneous forms of degree $2n$ on  
 $\R^d$. The only invariant forms are those   
proportional to the form $\langle v,v\rangle^n$, so  
both $\mbox{Sym}_{2n}\R^d$ and its dual space contain one copy of the trivial  
representation, and $\nu_{2n}(\sphere^{d-1})$ lies in the hyperplane of  
$\mbox{Sym}_{2n}\R^d$ defined by $\langle v,v\rangle^n = 1$.  
  
The dual cone to the Veronese orbitope  
$\,\mathcal{V}_{d,2n} =  
{\rm conv}(\nu_{2n}(\sphere^{d-1}))\,$ is the cone  
of non-negative forms of degree $2n$ in   
$\mbox{Sym}_{2n}(\R^d)^*$.   
See also \cite[Example (1.2)]{BB}.  
We write $\,\widehat{\mathcal{V}}_{d,2n}^\circ$ for the   
Veronese coorbitope cone consisting of non-negative forms.  
The cone $\widehat{\mathcal{V}}_{d,2n}^\circ$ of non-negative forms  
contains the cone $\mathcal{K}_{d,2n}$ of sums of squares, but when  
$d\geq 3$,  $2n\geq 4$, and $(d,2n)\neq(3,4)$, Hilbert~\cite{hilbert} showed  
that the inclusion is strict.  
We refer to \cite{Ble} for a recent study which compares the dimension of the   
faces of these cones.   
  
The cone $\mathcal{K}_{d,2n}$ is naturally a projection of   
the positive semidefinite cone.  
Hence its dual cone $\,\mathcal{C}_{d,2n} = \mathcal{K}_{d,2n}^\circ\,$  
is a spectrahedron: it can  be realized as   
the intersection of the positive semidefinite cone with a certain linear   
space of generalized Hankel matrices,   
discussed in detail in Reznick's book \cite{Rez}.   
This spectrahedron $\mathcal{C}_{d,2n}$ is strictly larger   
than the Veronese orbitope $\mathcal{V}_{d,2n}$ when $d\geq 3$,    
$2n\geq 4$, and $(d,2n)\neq(3,4)$.  
In fact,  $\mathcal{V}_{d,2n}$ is precisely the convex hull of   
the subset of extreme points in $\mathcal{C}_{d,2n}$ that have rank one.   
  
We present a detailed case study of the exceptional  
case of ternary quartics, when $\,(d,2n)=(3,4)$. The Veronese orbitope  
$\mathcal{V}_{3,4} = {\rm conv}(\nu_4(\R^3))$ is a $14$-dimensional  
convex body. Let $\widehat{\mathcal{V}}_{3,4}$ be the  
$15$-dimensional cone over the Veronese orbitope  
$\mathcal{V}_{3,4}$.  
As all non-negative ternary quartics are sums of squares, we have  
the following identities of cones:  
\[  
    \widehat{\mathcal{V}}_{3,4} \,=\, \mathcal{C}_{3,4}  
  \quad \hbox{and} \quad  
    \widehat{\mathcal{V}}_{3,4}^\circ \,=\, \mathcal{K}_{3,4}.   
\]  
We next present Reznick's spectrahedral representation of  
$\mathcal{V}_{3,4}$. For this, we identify  $\mbox{Sym}_4\R^3$ with its dual,  
and we  introduce coordinates  
$\l=(\l_\alpha)$ where the indices $\alpha$ are   
the exponents of monomials in  
variables $x,y,z$ of degree $4$.  
The ternary quartic corresponding to $\l$ is  
 \begin{equation}\label{Eq_Ter_Quartic}  
   q_\l\ =\ \sum_\alpha \tbinom{4}{\alpha}\l_\alpha x^{\alpha_1}y^{\alpha_2}z^{\alpha_3}\,,  
 \end{equation}  
where $\binom{4}{\alpha} = \frac{4 !}{  
\alpha_1 ! \alpha_2 !\alpha_3 !}$ is the multinomial coefficient.  
The inner product  
$\,  
   \langle q_\l, q_\mu\rangle\ =\ \sum_\alpha  \tbinom{4}{\alpha}\l_\alpha\mu_\alpha \,$  
is $SO(3)$-invariant.  
Given a ternary quartic $q_\l$ in the notation \eqref{Eq_Ter_Quartic}, we  
associate to it the following  
symmetric $6 \times 6$-matrix with Hankel structure   
as in \cite[eqn.~(5.25)]{Rez}:   
 \begin{equation}\label{Eq:Reznick_Hankel}  
   K_\l \quad = \quad   
  \begin{pmatrix}   
   \l_{400} & \l_{220} & \l_{202} & \l_{310} & \l_{301} & \l_{211}  \\   
   \l_{220} & \l_{040} & \l_{022} & \l_{130} & \l_{121} & \l_{031} \\   
   \l_{202} & \l_{022} & \l_{004} & \l_{112} & \l_{103} & \l_{013} \\   
   \l_{310} & \l_{130} & \l_{112} & \l_{220} & \l_{211} & \l_{121} \\   
   \l_{301} & \l_{121} & \l_{103} & \l_{211} & \l_{202} & \l_{112} \\   
   \l_{211} & \l_{031} & \l_{013} & \l_{121} & \l_{112} & \l_{022}     
   \end{pmatrix}.   
 \end{equation}  
  
\begin{theorem}   
 The  Veronese orbitope   
 $\,\mathcal{V}_{3,4}$ is a spectrahedron.  
 It consists of all positive semidefinite   
 Hankel matrices $K_\l$ as in   
$\eqref{Eq:Reznick_Hankel}$ that satisfy the equation   
 \begin{equation}\label{Eq:affine}  
  \l_{400} + \l_{040}  + \l_{004} +    
  2 \l_{220} + 2 \l_{202} + 2 \l_{022} \,\,=\,\,  1.   
 \end{equation}  
\end{theorem}   
  
\begin{proof} It is shown in~\cite[Ch.~5]{Rez}    
that the quartic $q_\l$ is non-negative if and  
 only if the Hankel matrix  $K_\l$ is positive semidefinite.  
 Furthermore, the equation~\eqref{Eq:affine} is the affine   
equation  $(x^2+y^2+z^2)^2 = 1$ which defines  
 the hyperplane containing the orbit  $\nu_4(\sphere^2)$.  
\end{proof}  
  
By contrast, the Veronese coorbitope is not a spectrahedron.  
  
\begin{theorem}\label{Th:SOS}  
  The convex cone of non-negative ternary quartics  
$\mathcal{K}_{3,4} = \widehat{\mathcal{V}}_{3,4}^\circ$ is not a  
  spectrahedron.  
  Its facets have dimension twelve and the intersection of any two facets is an  
  exposed face of dimension nine.  
  It also has maximal non-exposed faces of dimension nine.  
\end{theorem}   
  
\begin{proof}  
 We thank Greg Blekherman who explained this to us.  
 The cone $\mathcal{K}_{3,4}$ is full-dimensional in the $15$-dimensional   
space  $\mbox{Sym}_{2n}(\R^d)^*$.  
 Its facets come from its defining linear inequalities  
\[  
   \mathcal{K}_{3,4}\ =\  
    \bigl\{  
q\in\mbox{Sym}_{2n}(\R^d)^*\mid q(p)\geq 0\quad \forall p\in\R\PP^2  
\bigr\}\,.  
\]  
 For $p\in\R\PP^2$, let $F^p$ be the facet exposed by the inequality   
 $q(p)\geq 0$, which consists of those non-negative ternary quartics $q$  
 that vanish at $p$.  Since the boundary of   
$\mathcal{K}_{3,4}$ is $14$-dimensional and we have a  
 two-dimensional family of isomorphic   
facets, we see that each facet $F^p$ is  
 $12$-dimensional. A non-negative form that vanishes at $p\in\R\PP^2$  
 must also have its two partial derivatives (in local coordinates at $p$)  
 vanish, which gives three linear conditions on the facet $F^p$.  
 Concretely, if we take   
$p=[0:0:1]$ with local coordinates $x,y$, then the  
 constant and linear terms of the inhomogeneous   
quartic $q(x,y)$ must vanish.  
 Consequently,   
 \begin{equation}\label{Eq:homog_decomp}  
  q(x,y)\ =\ H(x,y)+C(x,y)+Q(x,y)\,,  
 \end{equation}  
 where $H$, $C$, and $Q$ are, respectively, the   
terms of degrees $2$, $3$, and $4$ in  
 $q$. These are binary forms. Their $3+4+5=12$   
coefficients parametrize the linear span of the facet $F^p$, showing again that $F^p$ is  
12-dimensional.   
 The quadratic form $H(x,y)$ is the Hessian of $q(x,y)$ at   
$p=[0:0:1]$ in these coordinates.  
  
 A form $q$ lies in the   
relative interior of the facet $F^p$ if and only if,  
 given any $q'$ which vanishes at $p$ along with its   
partial derivatives, so that it lies in the linear span of $F^p$, there  
 is an $\epsilon>0$ such that $q+\epsilon q'$ also lies in $F^p$.  
These conditions are equivalent to   
\begin{enumerate}  
 \item $q$ has no other zeros in $\R\PP^2$, and   
 \item the Hessian of $q$ at $p$ is a positive definite quadratic form.  
\end{enumerate}  
 A form $q\in F^p$ lies in the boundary of $F^p$ when one of   
these conditions fails, that is, either  
\begin{enumerate}  
 \item $q$ vanishes at a second point $p'\in\R\PP^2\setminus\{p\}$, or  
 \item the Hessian of $q$ has a double root at some point $r\in\R\PP^2$.  
\end{enumerate}  
 Faces of type (1) have the form $F^p\cap F^{p'}$.  
 These are nine-dimensional and occur in a four-dimensional   
family parametrized  
 by pairs of distinct points in $\R\PP^2$.  
 The union of all such faces  is a semialgebraic subset  
of dimension $9+4=13$ in the boundary of  
 $\mathcal{K}_{3,4}$.

 Faces of type (2) also have dimension nine.  
 The condition that the Hessian has a double   
root at a point $r\in\R\PP^1$ gives  
 two linear conditions on the coefficients of $q$.  
 There is an additional condition that the cubic part $C$ of  
 $q$ in \eqref{Eq:homog_decomp} also   
vanishes at $r$, for otherwise $q$ takes  
 negative values along the line through $p$ corresponding to $r$.  
 A face of type (2) is the limit of faces   
$F^p\cap F^{p'}$ of type (1) as  
 $p'$ approaches $p$ along the line corresponding to $r$.  
 These faces form a three-dimensional family on which $SO(3)$ acts   
faithfully and transitively.   
   
 Let us now examine the exposed faces of $\mathcal{K}_{3,4}$.  
 Let $\ell\in \widehat{\mathcal{V}}_{3,4}$ be a symmetric tensor in the  
 dual cone to $\mathcal{K}_{3,4}$.  
 Then $\ell$ is the sum of decomposable symmetric tensors,  
\[  
    \ell\ =\ \nu_4(p_1)+\dotsb+\nu_4(p_s)\,,  
\]  
 and so it supports the face $F^{p_1}\cap \cdots \cap F^{p_s}$.  
 Thus the exposed faces of $\mathcal{K}_{3,4}$ are intersections of  
 facets, and they consist   
of non-negative ternary quartics that vanish at a  
 given set of points.  
  
 Since the faces of type (2) are not of this form, they are not exposed.  
\end{proof}  
  
Our next agenda item is a discussion of  
the algebraic boundaries of the convex bodies and cones  
discussed above.  
The algebraic boundary of the cone $\widehat{\mathcal{V}}_{3,4}$ is   
characterized by the vanishing of the  
determinant of  the Hankel matrix $K_\l$ in~\eqref{Eq:Reznick_Hankel}.  
From this we conclude:  
  
\begin{corollary}  
The algebraic boundary of the   
 Veronese orbitope $\mathcal{V}_{3,4}$  
is the variety of dimension $13$ and degree six  
which is defined by  
the linear equation~\eqref{Eq:affine} and     
the Hankel determinant $\,{\rm det}(K_\l) = 0$.  
The extreme points of $\mathcal{V}_{3,4}$  
are precisely the Hankel matrices $K_\l$ of rank $1$.  
\end{corollary}  
  
We observed in the proof of Theorem~\ref{Th:SOS} that the boundary of $\mathcal{K}_{3,4}$  
consists of non-negative quartics $q$ that vanish at some point $p$ of $\R\PP^2$, and that the  
partial derivatives of $q$ necessarily also vanish at $p$.  
That is, the plane quartic curve defined by $q=0$ is singular at $p$.  
Thus the algebraic boundary of $\mathcal{K}_{3,4}$ consists of singular   
ternary quartics.  
Working in $\PP(\mbox{Sym}_4\C^3)\simeq\PP^{14}$,  
and its dual space of ternary quartics, this algebraic  
boundary is seen to be the dual variety to the Veronese surface   
which consists of rank $1$ Hankel matrices $K_\l$.  
  
\begin{corollary} \label{cor:discrorbi}  
The algebraic boundary of the   
coorbitope cone $\mathcal{K}_{3,4}$  
is an irreducible hypersurface of degree $27$.  
Its defining polynomial is the discriminant $\Delta_q$  
of the ternary quartic  
\begin{multline*}   
  \quad q(x,y,z) \,\, = \,\,  
  c_{400}   x^4  +  c_{310}   x^3 y +  c_{301}   x^3 z  +    
  c_{220}   x^2 y^2 +  c_{211}   x^2 y z +  c_{202}   x^2 z^2 +    
  c_{130} x y^3 \\ \qquad +  c_{121}   x y^2 z    +  c_{112}   x y z^2 +   
 c_{103}   x z^3 +    
  c_{040} y^4 + c_{031} y^3 z + c_{022} y^2 z^2 + c_{013} y z^3 + c_{004}   z^4  \,.  
  \quad   
\end{multline*}   
\end{corollary}  
  
The discriminant $\Delta_q$ is a homogeneous polynomial  
of degree $27$ in the $15$ indeterminates $c_{ijk}$.  
In what follows we shall present an explicit expression for $\Delta_q$.  
That expression will be derived   
from a beautiful classical formula due to Sylvester  
which can be found in Section 3.4.D,  
starting on page 118, of the book by Gel$'$fand,  
Kapranov, and Zelevinsky~\cite{GKZ}.  
% We note that our convention for dual and primal spaces is opposite   
% to that in~\cite{GKZ}.  
  
According to \cite[Prop.~1.7, page 434]{GKZ}, the   
discriminant $\Delta_q$ is proportional to the resultant $R_3(q_x,q_y,q_x)$   
of the three partial derivatives of the quartic $q$. Here $R_3$  
denotes the resultant of three ternary cubics, and the precise relation is  
 $\, \Delta_q \,= \, 4^{-7} \cdot R_3(q_x,q_y,q_x)$.  
  
We write $(\R^3)^*$ for the space of linear forms on   
$\R^3$, and we introduce the linear map  
 \begin{equation}\label{Eq:first_map}  
   T\ \colon\ (\R^3)^*\oplus(\R^3)^*\oplus(\R^3)^*\ \longrightarrow\   
    \mbox{Sym}_4(\R^3)^* \,, \,\,(f,g,h)\mapsto f q_x + g q_y + h q_z.   
 \end{equation}  
Next, for an exponent vector $\alpha=(\alpha_2,\alpha_2,\alpha_3)$ of degree  
$\alpha_1+\alpha_2+\alpha_3=2$ and any variable $t\in\{x,y,z\}$,   
we choose a decomposition of the cubic partial derivative   
 \begin{equation}\label{Eq:decomp}  
   q_t\ =\ x^{\alpha_1+1}P^{(t)}_\alpha\,+\,  
           y^{\alpha_2+1}Q^{(t)}_\alpha\,+\,  
           z^{\alpha_3+1}R^{(t)}_\alpha\,,  
 \end{equation}  
where $P^{(t)}_\alpha$, $Q^{(t)}_\alpha$, and $R^{(t)}_\alpha$ are forms   
of degree $2-\alpha_1$, $2-\alpha_2$, and $2-\alpha_3$, respectively. Then  
\[  D_\alpha\ =\ \det\left( \begin{matrix}  
     P^{(x)}_\alpha & Q^{(x)}_\alpha & R^{(x)}_\alpha \\  
     P^{(y)}_\alpha & Q^{(y)}_\alpha & R^{(y)}_\alpha \\  
     P^{(z)}_\alpha & Q^{(z)}_\alpha & R^{(z)}_\alpha \end{matrix}\right)\,\]  
is a quartic polynomial.  
Finally, we  define a linear map   
$\,D\colon \mbox{Sym}_2\R^3 \to\mbox{Sym}_4(\R^3)^* $  
by sending $\delta_\alpha\mapsto D_\alpha$, where $\{\delta_\alpha\}$ is the basis dual to the  
monomial basis of $\mbox{Sym}_2(\R^3)^*$.  
  
%%%%%%%%%%%%%%%%%%%%%%%%%%%%%%%%%%%%%%%%%%%%%%%%%%%%%%%%%%%%%%%%%%%%%%%%%%%%%%%%%%%%%%%%%%%   
\begin{proposition} {\rm (Sylvester \cite[\S3.4.D]{GKZ})}  
 The discriminant $\Delta_q$ is proportional to the  
 resultant of the ternary cubics  
$q_x,q_y,q_z$ which is equal to the determinant of the linear map  
\[  
   T\oplus D\ \colon\   
   (\R^3)^*\oplus(\R^3)^*\oplus(\R^3)^*\oplus\mbox{\rm Sym}_2(\R^3)^*\   
   \longrightarrow\ \mbox{\rm Sym}_4(\R^3)^*\,.  
\]  
This is an irreducible homogeneous polynomial of degree $27$ in the $15$ coefficients $c_{ijk}$.  
\end{proposition}  
  
We write this map explicitly as a $15\times 15$ matrix $\mathcal{D}(q)$   
whose rows are indexed by the $15$ monomials $x^iy^jz^k$ of degree   
$i+j+k=4$ and whose columns are indexed by  
$15$ auxiliary quartics, and whose entry in a given row   
and column is the coefficient of that  monomial in that auxiliary quartic.   
Nine of the quartics come from the map $T$ in~\eqref{Eq:first_map}.  
They are  
\[   xq_x\,,\  yq_x\,,\  zq_x\,,\   
   xq_y\,,\  yq_y\,,\  zq_y\,,\   
   xq_z\,,\  yq_z\,,\  zq_z\,.  
\]  
The other six are the polynomials   
$D_{002}$, $D_{020}$, $D_{200}$, $D_{110}$, $D_{101}$, and $D_{011}$   
from $D$.  
We  only describe $D_{002}$ and $D_{110}$ as the others may be recovered  
from these by symmetry.   
  
For $D_{002}$, note that each partial derivative of $q$   
has six terms divisible by $x$, three divisible by $y$  
and not $x$, and a unique term involving $z^3$.  
This leads to a decomposition~\eqref{Eq:decomp}, and   
$D_{002}$ is the determinant of the $3\times 3$ matrix: 
\[  
  \begin{bmatrix}  
    4c_{400}x^2 + 3c_{310}xy   
+ 3c_{301}xz + 2c_{220}y^2 + 2c_{211}yz + 2c_{202}z^2 \!&\!  
     c_{130}y^2 +  c_{121}yz +  c_{112}z^2 \!&\! c_{103} \\  
   \!  
     c_{310}x^2 + 2c_{220}xy +  c_{211}xz + 3c_{130}y^2 + 2c_{121}yz +  c_{112}z^2 \!&\!  
    4c_{040}y^2 + 3c_{031}yz + 2c_{022}z^2 \!&\! c_{013} \\  
   \!  
     c_{301}x^2 +  c_{211}xy + 2c_{202}xz +  c_{121}y^2 + 2c_{112}yz + 3c_{103}z^2 \!&\!  
     c_{031}y^2 + 2c_{022}yz + 3c_{013}z^2 \!&\! 4c_{004}  
  \end{bmatrix}  
\]  
By a similar reasoning, we find that  
 $D_{110}$ is the determinant of the   
$3\times 3$ matrix:  
\[  
  \begin{bmatrix}  
    4c_{400}x  + 3c_{310}y  + 3c_{301}z \!&\!  
    2c_{220}x  +  c_{130}y  +  c_{121}z \!&\!  
    2c_{211}xy + 2c_{202}xz +  c_{112}yz +  c_{103}z^2  
   \\  
     c_{310}x  + 2c_{220}y  +  c_{211}z  \!&\!  
    3c_{130}x  + 4c_{040}y  + 3c_{031}z  \!&\!  
    2c_{121}xy +  c_{112}xz + 2c_{022}yz +  c_{013}z^2  
   \!\\  
     c_{301}x  +  c_{211}y  + 2c_{202}z  \!&\!  
     c_{121}x  +  c_{031}y  + 2c_{022}z   \!&\!  
    2c_{112}xy + 3c_{103}xz + 3c_{013}yz + 4c_{004}z^2      
  \end{bmatrix}  
\]  
This concludes our discussion of the  
algebraic boundary of the coorbitope  
cone $\mathcal{K}_{3,4}$.  
  
We close with the remark that the notations $\mathcal{K}_{\bdt,\bdt} $   
and $ \mathcal{C}_{\bdt,\bdt}$ are consistent with those  
used in the paper \cite{SU}    
where these cones consist of concentration matrices   
and sufficient statistics of a certain  Gaussian model.

\section{Grassmann Orbitopes}  
  
The {\em Grassmann orbitope} $\mathcal{G}_{d,n}$ is the convex hull  
of the Grassmann variety of oriented $d$-dimensional linear subspaces of $\R^n$  
in its Pl\"ucker embedding in the unit sphere in $\wedge_d \R^n$.  
Equivalently, this is the highest weight orbitope for the group $SO(n)$  
acting  on   $\wedge_d \R^n$:  
$$ \qquad \qquad \mathcal{G}_{d,n} \,\, = \,\, {\rm conv}( SO(n)   
\cdot e_{12\cdots d}) \qquad  
\hbox{where $\,e_{12 \cdots d} \,=\, e_1 \wedge e_2 \wedge \cdots \wedge e_d \, \in \, \wedge_d \R^n$} . $$  
Faces of the Grassmann orbitope are of considerable interest in differential geometry since,  
according to the Fundamental Theorem of Calibrations, they correspond to  
area-minimizing $d$-dimensional submanifolds of $\R^n$. References to this  
subject include the  
seminal article on calibrated geometries by Harvey and Lawson \cite{HL1982}   
and the beautiful expositions by Morgan \cite{HarMor86, Mor85}.  
In this section we review basic known facts about $\mathcal{G}_{d,n}$ and we   
initiate its study  
from the perspectives of combinatorics, semidefinite programming, and  
algebraic geometry.  
  
Vectors in $\wedge_d \R^n$ are written in terms of Pl\"ucker coordinates  
relative to the standard basis:  
$$ p \,\,\,\,=\, \sum_{1 \leq i_1 < \cdots < i_d \leq n} \!\!\! p_{i_1 i_2 \cdots i_d} \,  
e_{i_1} \wedge e_{i_2} \wedge \cdots \wedge e_{i_d}. $$  
The Pl\"ucker vector $p$ lies in the {\em oriented Grassmann variety} if and only if it  
is decomposable, i.e. $\,p = u_1 \wedge u_2 \wedge  \cdots \wedge u_d$  
for some pairwise orthogonal subset $\{u_1,u_2,\ldots,u_d \}$ of $  
\sphere^{n-1}$.  
This happens if and only if $p$ lies in the unit sphere in $\wedge_d \R^n$  
and satisfies all quadratic Pl\"ucker relations, the relations among  
the $d {\times} d$-minors of a $d {\times} n$-matrix. These relations  
generate the prime ideal called the {\em Pl\"ucker ideal} $I_{d,n}$. 
Thus the oriented Grassmann variety is the algebraic  
subvariety of $\wedge_d \R^n$ defined by the ideal  
\begin{equation}  
\label{eq:PluckerIdeal}  
 I_{d,n} \,\,+\,\  
\bigl\langle \, 1 \, -   
  \!\!\!\! \sum_{1 \leq i_1 < \cdots < i_d \leq n}  p_{i_1 i_2 \cdots i_d}^2 \, \bigr\rangle.  
\end{equation}  
The convex hull of that real algebraic variety is the  
$\binom{n}{d}$-dimensional Grassmann orbitope~$\mathcal{G}_{d,n}$.

\begin{example}[$d{=}2, n{=}4$] \label{ex:G24}  
The Grassmann orbitope  
$\mathcal{G}_{2,4}$ is the convex hull of the variety~of  
\begin{equation}  
\label{eq:PluckerIdeal24}  
\bigl\langle \,  
p_{12} p_{34} - p_{13} p_{24} + p_{14} p_{23} \,,\,\,  
p_{12}^2 + p_{13}^2 + p_{14}^2 +   
p_{23}^2 + p_{24}^2 + p_{34}^2 -1\, \bigr\rangle.   
\end{equation}  
As suggested by \cite[Proposition~3.2]{Mor85}, we  perform the orthogonal change of
coordinates    
$$   
\begin{matrix}  
u = \frac{1}{\sqrt{2}}(p_{12}+p_{34}), &  
v = \frac{1}{\sqrt{2}}(p_{13}-p_{24}), &  
w = \frac{1}{\sqrt{2}}(p_{14}+p_{23}), \\ \rule{0pt}{14pt} 
x = \frac{1}{\sqrt{2}}(p_{12}-p_{34}), &  
y = \frac{1}{\sqrt{2}}(p_{13}+p_{24}) ,&  
z = \frac{1}{\sqrt{2}}(p_{14}-p_{23}) .  
\end{matrix}   
$$  
This is simultaneous rotation by $\pi/4$ in each of the coordinate planes
spanned by the pairs
$( p_{12},p_{34})$, 
$( p_{13},p_{24})$, and 
$( p_{14},p_{23})$.
In these new coordinates, the prime ideal (\ref{eq:PluckerIdeal24}) equals  
$$ \bigl\langle \,u^2+ v^2 + w^2 - \frac{1}{2} \, ,\,\,  
x^2+ y^2 + z^2 - \frac{1}{2} \,\,\bigr\rangle.   
$$  
This reveals that  $\mathcal{G}_{2,4}$ is the direct product of  
two three-dimensional balls of radius $1/\sqrt{2}$. \qed  
\end{example}  
  
We next examine the case $d = 2$ and arbitrary $n$.  
The vectors $p$ in $ \wedge_2 \R^n$ can be identified with skew-symmetric  
$n {\times} n$-matrices, and this brings us back to the  
orbitopes in Section 3.2.

\begin{corollary} \label{rmk:GrassInSec3}  
The Grassmann orbitope $\mathcal{G}_{2,n}$ coincides with the  
skew-symmetric Schur-Horn orbitope of a   
skew-symmetric matrix $N \in \wedge_2 \R^n$ having rank two and  
 $\Lambda(N) = (1,0,0,\ldots,0)$.  
\end{corollary}  
  
If $p$ is a real skew-symmetric matrix whose eigenvalues  are   
$\pm i \hl_1,\ldots,\pm i \hl_k$,  
where $i = \sqrt{-1}$, then  
the matrix $\,i \cdot p\,$ is Hermitian  
and its eigenvalues are the real numbers $\pm \hl_1,\ldots,\pm \hl_k$.  
Recall that the operator $\mathcal{L}_k$ computes the  
$k$-th additive compound matrix of a given matrix.  
  
\begin{theorem} \label{thm:G2n}  
Let $n \geq 5$ and $k = \lfloor n/2 \rfloor$.  The Grassmann orbitope equals 
the spectrahedron  
\begin{equation}  
\label{eq:spectraG2n}  
 \mathcal{G}_{2,n} \,\, = \,\, \bigl\{ \, p \in \wedge_2 \R^n \,:\,  
{\rm Id}_{\binom{n}{k}} -  
\Sfunc_{k}(i \cdot p) \,\succeq \, 0 \,\bigr\}.   
\end{equation}  
Its algebraic boundary $\partial_a \mathcal{G}_{2,n}$ is an irreducible 
hypersurface of degree  $2^k$, defined by a factor of the determinant 
of the matrix ${\rm Id}_{\binom{n}{k}} - \Sfunc_{k}(i \cdot p)$.  
The proper faces of $\mathcal{G}_{2,n}$ are $SU(m)$-orbitopes for $1 
\le m \le k$. Every face $F$ is associated with an even-dimensional 
subspace $V_F$ equipped with an orthogonal complex structure and the extreme 
points of $F$ correspond to complex lines in $V_F$. 
\end{theorem}  
  
Everything in this theorem is also true for the small cases $n =  
3,4$, with the exception that the quartic hypersurface $\partial_a  
\mathcal{G}_{2,4}$ is not irreducible, as was seen in  Example \ref{ex:G24}.  
In light of Theorem~\ref{thm:facelatticeSkewSH},  
the reducibility of $\mathcal{G}_{2,4}$ arises because  
the two-dimensional crosspolytope equals the square, which   
decomposes as a  Minkowski sum of two line segments.  
This may also be seen as the stabilizer of a decomposable tensor in $SO(4)$ is $\pm I$, 
where $I$ is the identity matrix, and $SO(4)/\{\pm I\}\simeq SO(3)\times SO(3)$, so 
$\mathcal{G}_{2,4}$ is also an orbitope for $SO(3)\times SO(3)$. 
  
\begin{proof}  
We begin with the last statement about the face lattice of $\mathcal{G}_{2,n}$.  
This result is well-known in the theory of calibrations, where it is usually phrased  
as follows: {\em every face of the Grassmannian of two-planes in  
$\R^n$ consists of the complex lines in some $2m$-dimensional subspace  
of\/ $\R^n$ under some orthogonal complex structure}. See~\cite[\S 1.1]{Mor85}.  
  
To derive the spectrahedral representation  
(\ref{eq:spectraG2n}), we note that the eigenvalues of  
$\Sfunc_{k}(i \cdot p)$ are the sums of any $k$ distinct  
numbers of the eigenvalues of the skew-symmetric matrix $p$. These   
are   $ -\hl_1,\ldots -\hl_k,\hl_1 ,\ldots,\hl_k$  
if $n$ is even and  
$ -\hl_1,\ldots, -\hl_k,0,\hl_1,\ldots,\hl_k$  
if $n$ is odd.  
In light of Corollary \ref{rmk:GrassInSec3}, we can apply the  
results in Section 3.2 to conclude that $p$ lies in $\mathcal{G}_{2,n}$  
if and only if  $\,\pm \hl_1 \pm \hl_2 \pm \cdots \pm \hl_k \leq 1\,$  
 for all choices of signs. In terms of polyhedral geometry,  
 this condition means that the vector  
 $\Lambda(p) = (\hl_1,\hl_2,\ldots,\hl_k)$ lies in the  
crosspolytope $\Lambda(\mathcal{G}_{2,n})$.  
Since all $\binom{n}{k}$ eigenvalues of $\Sfunc_{k}(i \cdot p)$  
are bounded above by the maximum of the $2^k$  
special eigenvalues $\,\pm \hl_1 \pm \hl_2 \pm \cdots \pm \hl_k$,  
we conclude that $p \in \mathcal{G}_{2,n}$ if and only if  
$\,{\rm Id}_{\binom{n}{k}} - \Sfunc_{k}(i \cdot p) \,\succeq \, 0 $.  
  
To compute the algebraic boundary $\,\partial_a \mathcal{G}_{2,n}\,$  
we consider the expression  
$$  
\prod_{\sigma \in \{-1,+1\}^k}  (1+\sigma_1 \hl_1 + \sigma_2 \hl_2 + \cdots + \sigma_k \hl_k) .$$  
This is a symmetric polynomial of degree $2^{k-1}$ in the  
squared eigenvalues $\hl_1^2, \hl_2^2, \ldots,\hl_k^2$,  
and hence it can be written as a polynomial in the coefficients of the characteristic polynomial  
$$ {\rm det}\bigl(i\cdot p - x \cdot {\rm Id}_{n} \bigr) \,\, = \,\,  
x^{n \,{\rm mod} \, 2} \cdot    
(x^2 - \hl_1^2)  
(x^2 - \hl_2^2) \cdots  
(x^2 - \hl_k^2).  
 $$  
The resulting polynomial has degree $2^k$ in the entries  
$p_{ab}$ of the matrix $p $, and it vanishes on the boundary of  
the orbitope $\mathcal{G}_{2,n}$.   
 It can be checked that it is irreducible for $n \geq 5$.  
\end{proof}  
  
\begin{example}  
Let $n = 6$ and consider the characteristic polynomial  
of our Hermitian matrix:  
$$   
{\rm det}  
\begin{pmatrix}  
-x &   i p_{12} & i p_{13} & i p_{14} & i p_{15} & i p_{16} \\  
- i p_{12} & -x & i p_{23} & i p_{24} & i p_{25} & i p_{26} \\  
- i p_{13} & - i p_{23} & -x & i p_{34} & i p_{35} & i p_{36} \\  
- i p_{14} & - i p_{24} & - i p_{34} & - x &  i p_{45} & i p_{46} \\  
- i p_{15} & - i p_{25} & - i p_{35} &   i p_{45} & - x  & i p_{56} \\  
- i p_{16} & - i p_{26} & - i p_{36} &   i p_{46} & - i p_{56} & -x    
\end{pmatrix} \quad   
\begin{matrix} = &&  
x^6 + a_4 x^4 + a_2 x^2 + a_0 \\   
& & \\  
= &&   
(x^2 - \hl_1^2)(x^2 - \hl_2^2)(x^2 - \hl_3^2)  
\end{matrix}  
$$  
The algebraic boundary of the Grassmann orbitope  $ \mathcal{G}_{2,6}$   
is derived from the   
polynomial  
$$   
\prod_{\sigma \in \{\pm1\}^3} \!\!\!\!  
(1 + \sigma_1 \hl_1 +  \sigma_2 \hl_2 + \sigma_3 \hl_3 )  
\ = \  
a_4^4  
{+}4 a_4^3  
{-}8 a_4^2 a_2  
{+}16 a_2^2  
{-}16 a_4 a_2  
{+}6 a_4^2  
{+}64 a_0  
{-}8 a_2  
{+}4 a_4  
{+}1\,.  
$$  
We rewrite this expression in terms of the $15$ unknowns $p_{ij}$  
to get  an irreducible polynomial of degree $8$ with $10791$ terms.  
This is the defining polynomial of the hypersurface  $\partial_a \mathcal{G}_{2,6}$.  
For $n=7$ we use the same polynomial  
in $a_0,a_2,a_4$ but now there is one more eigenvalue $0$.  
The defining polynomial  
of $\partial_a \mathcal{G}_{2,7}$ has  
$44150$ terms of degree $8$ in the $21$ matrix entries $p_{ij}$.  
\qed  
\end{example}  
  
We now come to the  harder case $d=3, n=6$.  
The  Grassmann orbitope $\mathcal{G}_{3,6}$ is a $20$-dimensional  
convex body. Its facial structure   
was determined by Dadok and Harvey in \cite{dadok},  
and independently by Morgan in \cite{Mor85}.  
We present their well-known results in our language.  
  
\begin{theorem} \label{thm:morgan36}  
The orbitope $\mathcal{G}_{3,6}$ has  
three classes of positive-dimensional exposed faces:  
\begin{enumerate}  
\item  
one-dimensional faces from  
pairs of subspaces that satisfy the angle condition (\ref{eq:threeangles});  
\item  
three-dimensional faces that arise as $SO(3)$-orbitopes and these are round $3$-balls;  
\item  
$12$-dimensional faces that are   
$SU(3)$-orbitopes, from the  
 Lagrangian Grassmannian.  
 \end{enumerate}  
 \end{theorem}  
  
Harvey and Morgan \cite{HarMor86} extended these results  
to $\mathcal{G}_{3,7}$. We will here focus on $d=3, n=6$.  
Our first goal is to explain  
 the faces described in Theorem \ref{thm:morgan36}, starting with the edges.  
 Let $L$ and $L'$ be three-dimensional linear subspaces in $\R^6$  
 with corresponding unit length Pl\"ucker vectors  $p$ and $p'$.  
 We define $\theta_1 (L,L')$ to be the minimum angle between any  
 unit vector $v_1 \in L$ and any unit vector $w_1 \in L'$.  
 Next, $\theta_2(L,L')$ is the minimum angle  
 between two unit vectors $v_2 \in L$ and $w_2 \in L'$ such that  
 $v_2 \perp v_1$ and $w_2 \perp w_1$. Finally, we define  
 $\theta_3(L,L')$  to be the angle  
 between two unit vectors $v_3 \in L$ and $w_3 \in L'$ such that  
 $v_3 \perp \{v_1,v_2\}$ and $w_3 \perp \{ w_1,w_2\}$.  
We refer to \cite[Lemma 2.3]{Mor85} for the fact that  
$\theta_2$ and $\theta_3$ are well defined by this rule.  
 The angle condition referred to in part (1) of Theorem  \ref{thm:morgan36}  
 is the inequality  
  \begin{equation}  
 \label{eq:threeangles}  
 \theta_3 (L,L') \,<\, \theta_1(L,L') + \theta_2 (L,L').  
 \end{equation}  
  The convex hull of $p$ and $p'$ is an exposed edge of $\mathcal{G}_{3,6}$  
  if and only if the condition (\ref{eq:threeangles}) holds.  
   
 The maximal facets in part (3) of Theorem \ref{thm:morgan36} are  
 known to geometers  as {\em special Lagrangian facets},  
 and we represent them as orbitopes as follows.  
 The {\em special unitary group} $SU(3)$ consists of  
 all complex $3 {\times} 3$-matrices $U$ with ${\rm det}(U)= 1$  
 and $U U^* = {\rm Id}_3$. If $A = {\rm re}(U)$ and $B = {\rm im}(U)$,  
 so that $U = A + iB$, then we can identify $U$ with the real $6 {\times} 6$-matrix  
 $$ \tilde U \,\,\,  = \,\,\,  \begin{pmatrix}  
\,A & -B \, \\  
 \,B &\phantom{-} A \,\end{pmatrix} .$$  
  Note that this matrix lies in $SO(6)$ if and only if  
 $A A^T +B B^T = {\rm Id}_3$,  
 $\,A B^T = B A^T$, and ${\rm det}(A+iB) =1$.  
Hence the transformation $U \mapsto \tilde U$ realizes  
 $SU(3)$ as a subgroup of $SO(6)$.   
  
The orbit $\,SU(3) \cdot e_1 {\wedge} e_2 {\wedge} e_3\,$  
is known as the {\em special Langrangian Grassmannian}.  
This is a five-dimensional real algebraic variety in the  
unit sphere in $\wedge_3 \R^6$.  We call its convex hull  
$$\,\mathcal{SL}_{3,6} \,=\, {\rm conv} (SU(3) \cdot e_1 {\wedge} e_2 {\wedge} e_3) $$  
the {\em special Lagrangian orbitope}.  
It is obtained from the $20$-dimensional Grassmann orbitope  
$\mathcal{G}_{3,6}$  by maximizing the linear function  
 $\, p_{123} - p_{156} + p_{246} - p_{345} $, which takes  
value $1$ on that face. The face $\mathcal{SL}_{3,6}$  
 is $12$-dimensional  
because it satisfies  the two affine equations  
$$ p_{123} - p_{156} + p_{246} - p_{345}   
\, = \, {\rm re}( A + iB) \,=\,1  
\,\,\, \hbox{and} \,\,\,  
 p_{126} - p_{135} + p_{234} - p_{456}    
\, = \, {\rm im}( A + iB)   
\,=\, 0 $$  
plus the six independent linear equations  
that cut out the Langrangian Grassmannian:  
$$  
p_{125} + p_{136} \,=\,  
p_{134} + p_{235} \,=\,  
p_{124} - p_{236} \,=\,  
p_{146} + p_{256} \,=\,  
p_{245} + p_{346} \,=\,  
p_{145} - p_{356} \,\, =\,\, 0.  
$$  
  
  The exposed faces of type (2) in Theorem \ref{thm:morgan36}  
 are not facets. Each of them is the intersection of two special Lagrangian facets.  
 For example, consider the two linear functionals  
 $$   
 \phi_+ \, = \, p_{123} - p_{156} + p_{246} - p_{345}  
\quad \hbox{ and } \quad  
 \phi_- \, = \, p_{123} - p_{156} - p_{246} + p_{345}, $$  
 which were discussed in \cite[\S 4.4]{Mor85}. Each of them supports a  
 special Lagrangian facet. The intersection of these two   
 $14$-dimensional facets is a three-dimensional ball.   
Indeed, the linear functional $\frac{1}{2} ( \phi_+ + \phi_-) = p_{123} - p_{156}$   
is bounded above by $1$ on $\mathcal{G}_{3,6}$, and the subset  
at which the value equals $1$ is contained in the linear span of  
the six basis vectors $e_1 \wedge e_i \wedge e_j$ where  
$i,j \in \{2,3,5,6\}$. In the intersection of this linear space with   
$\mathcal{G}_{3,6}$ we find the Grassmann orbitope  
 $\mathcal{G}_{2,4}$ from Example \ref{ex:G24}, with   
    our linear functional  $p_{123} - p_{156}$ being represented by  
  the scaled coordinate $\sqrt{2} x$.  
  The face of $\mathcal{G}_{2,4}$ where $\sqrt{2} x$ attains  
  its maximal value $1$ is a $3$-ball.  
  This concludes our discussion of the   
  census of faces given in Theorem \ref{thm:morgan36}.  
    
  \smallskip  
  
  A natural question that arises next is  
     whether the Grassmann orbitope $\mathcal{G}_{3,6}$  
or its dual body  $\mathcal{G}_{3,6}^\circ $ can be represented  
  as a spectrahedron. It turns out that the answer is negative.  
  
\begin{theorem} \label{thm:G36bad}  
The Grassmann coorbitope  $\mathcal{G}_{3,6}^\circ$  
has extreme points that are not exposed.  
The Grassmann orbitope  $\mathcal{G}_{3,6}$  
has edges that are not exposed. Neither of them  
is a spectrahedron.  
\end{theorem}  
  
\begin{proof}  
The first assertion was proved by  
Dadok and Harvey in \cite[Theorem 7]{dadok}.  
For the second assertion we proceed as follows.  
We apply the technique in \cite[Lemma 2.2]{Mor85}  
and restrict to the linear subspace   
$\,V = \R \{e_{123},e_{126},e_{135},e_{234},e_{156},e_{246},e_{345},e_{456}\}  
$ of $\wedge_3 \R^6$.  
The intersection of $\mathcal{G}_{3,6} \cap V$ is  
the $SO(2) \times SO(2) \times SO(2)$-orbitope of  
 $e_{123}$, where the action is by unitary  
diagonal $3 {\times} 3$-matrices,  
 while the intersection  $\mathcal{SL}_{3,6} \cap V$   
is the $SO(2) \times SO(2)$-orbitope of $ e_{123}$,  
where the action is by special unitary  
 diagonal $3 {\times} 3$-matrices, as described in \cite[\S 4.3]{Mor85}.  
  
We claim that the  $SO(2) \times SO(2)$-orbitope  
$\mathcal{SL}_{3,6} \cap V$ is $2$-neighborly. This can be seen by   
examining the bivariate  
 trigonometric polynomials of the form  
$$ \begin{matrix}  
f \ = & \mbox{\ \quad}
\ x_{123} \cdot {\rm cos}(\alpha)  {\rm cos}(\beta)  {\rm cos}(-\alpha -\beta)  
\,\,+\,\,x_{126} \cdot {\rm cos}(\alpha)  {\rm cos}(\beta)  {\rm sin}(-\alpha -\beta)\\ &  
\,\,-\,\,x_{135}\cdot {\rm cos}(\alpha)  {\rm sin}(\beta)  {\rm cos}(-\alpha -\beta)   
\,\,+\,\,x_{234} \cdot {\rm sin}(\alpha)  {\rm cos}(\beta)  {\rm sin}(-\alpha -\beta) \\ &  
\,\,+\,\,x_{156} \cdot {\rm cos}(\alpha)  {\rm sin}(\beta)  {\rm sin}(-\alpha -\beta)   
\,\,-\,\,x_{246} \cdot {\rm sin}(\alpha)  {\rm cos}(\beta)  {\rm sin}(-\alpha -\beta) \\ &  
\,\,+\,\,x_{345} \cdot {\rm sin}(\alpha)  {\rm sin}(\beta)  {\rm cos}(-\alpha -\beta)  
\,\,-\,\,x_{456} \cdot {\rm sin}(\alpha)  {\rm sin}(\beta)  {\rm sin}(-\alpha -\beta).  
\end{matrix}  
$$  
Indeed, for any choice of  $(\alpha_1, \beta_1) $  
and $(\alpha_2,\beta_2)$, we can find  
 eight real coefficients $\,x_{ijk}\,$  
such that $f$ and its derivatives vanish at  
  $(\alpha_1, \beta_1) $  
and $(\alpha_2,\beta_2)$  
but $f$ is strictly positive elsewhere.  
  
The results in \cite{Mor85} imply that every 
exposed edge of $\mathcal{G}_{3,6} \cap V$ is  also 
an exposed edge of $\mathcal{G}_{3,6}$. We believe 
that Morgan's technique can be adapted to show that 
{\em every} exposed edge of $\mathcal{SL}_{3,6} \cap V$ is 
an exposed edge of $\mathcal{SL}_{3,6}$. For the purpose 
of proving Theorem  \ref{thm:G36bad}, however, we only need to identify 
{\em one} exposed edge of $\mathcal{SL}_{3,6} \cap V$ that is 
an exposed edge of $\mathcal{SL}_{3,6}$. 
 
The claim that such exposed edges exist can be derived 
from \cite[Theorem 12 (i)]{dadok}. With the help of Philipp Rostalski, 
we also obtained a computational proof of that claim. This was done as follows. 
 We selected various 
random choices of points  $(\alpha_1, \beta_1) $ and $  
(\alpha_2,\beta_2)$ 
in $\R^2$. Each choice specifies two 
three-dimensional subspaces $L$ and $L'$ of $\R^6$  for which equality   
holds in 
the angle condition (\ref{eq:threeangles}). The corresponding 
line segment is not an exposed edge of $\mathcal{G}_{3,6}$. 
 
To show that the line segment between $L$ and $L'$ is 
exposed in $\mathcal{SL}_{3,6}$, we use a technique from 
semidefinite programming. 
First we compute the eight coefficients $x_{ijk}$ 
of the supporting function $f$ as above.  This gives 
a linear function $\sum x_{ijk} p_{ijk}$ on $\wedge_3 \R^6$. 
We then run a first-order Lasserre relaxation 
(cf.~\cite{Lasserre}) to minimize 
this linear function subject to the linear and quadratic constraints 
that cut out the special Lagrangian Grassmannian. 
The optimal value is zero, the optimal 
Lasserre moment matrix  has rank two, and its image in 
$\wedge_3 \R^6$ lies in the relative interior 
of the line segment between $L$ and $L'$. 
We then re-optimize for various perturbations of 
the linear function $\sum x_{ijk} p_{ijk}$. The 
output of each run is a rank one moment matrix which certifies 
either $L$ or $L'$ as optimal solution of the optimization problem. 
This proves that the face of  $\mathcal{SL}_{3,6}$ 
exposed by $\sum x_{ijk} p_{ijk}$ is the line segment between $L$ and~$L'$. 
\end{proof} 
 
Theorem \ref{thm:G36bad}  shows that Grassmann orbitopes 
are generally not spectrahedra. We do not know whether they are 
%{\em spectrahedral shadows}, that is, 
linear projections of spectrahedra. 
 
%%%%%%%%%%%%%%%%%%%%%%%%%%%%%%%%%%%%%%%%%%%%%%%%%%%%%%%%%%%%%%%%%%%%%%%%%%%%%   
\def\cprime{$'$}
\providecommand{\bysame}{\leavevmode\hbox to3em{\hrulefill}\thinspace}
\providecommand{\MR}{\relax\ifhmode\unskip\space\fi MR }
% \MRhref is called by the amsart/book/proc definition of \MR.
\providecommand{\MRhref}[2]{%
  \href{http://www.ams.org/mathscinet-getitem?mr=#1}{#2}
}
\providecommand{\href}[2]{#2}

\end{document}